\documentclass[leqno]{amsart}
\usepackage[numbers]{natbib}
\usepackage{amssymb}

\renewcommand{\widehat}{\hat}
\usepackage{url}
\usepackage[colorlinks={true},
backref,pdfstartview={FitBH},pdfview={FitBH},bookmarks={true}]{hyperref}
\global\hbadness=500
\global\tolerance=200
\newtheorem{theorem}{Theorem}[section]
\newtheorem{lemma}[theorem]{Lemma}

\newtheorem{question}[theorem]{Question}

\newtheorem{corollary}[theorem]{Corollary}
\theoremstyle{definition}
\newtheorem{definition}[theorem]{Definition}

\theoremstyle{remark}

\newtheorem{remark}[theorem]{Remark}
\numberwithin{equation}{theorem}

\newcommand{\ce}{computably enumerable }
 
\newcommand{\join}{\oplus}
\renewcommand{\phi}{\varphi}
\newcommand{\gen}[1] {\langle #1\rangle}

\newcommand{\Dhat}{\widehat{D}}

\newcommand{\Rhat}{\widehat{R}}
\newcommand{\Hhat}{\widehat{H}}\newcommand{\Ehat}{\widehat{E}}

\title[$\mathcal{D}  $-maximal sets]{$\mathcal{D}$-maximal sets}

\author[P.\ Cholak]{Peter~A.~Cholak}

\address{Department of Mathematics\\ University of Notre Dame\\
  Notre Dame, IN 46556-5683}

\email{Peter.Cholak.1@nd.edu}

\urladdr{http://www.nd.edu/~cholak}

\author[P.\ Gerdes]{Peter Gerdes}
\email{gerdes@invariant.org}

\author[K.\ Lange]{Karen Lange}
\address{Department of Mathematics\\ Wellesley College\\
  Wellesley, MA 02482}
  \email{karen.lange@wellesley.edu}
  \urladdr{http://palmer.wellesley.edu/~klange/}

\thanks{This research was partially done while Cholak participated in
  the Buenos Aires Semester in Computability, Complexity and
  Randomness, 2013.
  Lange was partially supported by NSF DMS-0802961 and NSF
  DMS-1100604. }

\date{\today}

\begin{document}

 \begin{abstract}
   Soare \cite{Soare:74} proved that the maximal sets form an orbit in
   $\mathcal{E}$.  We consider here {\em $\mathcal{D}$-maximal sets},
   generalizations of maximal sets introduced by Herrmann and Kummer
   \cite{MR1264963}.  Some orbits of $\mathcal{D}$-maximal sets are
   well understood, e.g., hemimaximal sets \cite{DSAutsOrbits}, but
   many are not.  The goal of this paper is to define new invariants
   on computably enumerable sets and to use them to give a complete
   nontrivial classification of the $\mathcal{D}$-maximal sets.
   Although these invariants help us to better understand the
   $\mathcal{D}$-maximal sets, we use them to show that several
   classes of $\mathcal{D}$-maximal sets break into infinitely many
   orbits.
 \end{abstract}
 \maketitle

 \section{Introduction}

 Let $\mathcal{E}$ denote the structure of computably enumerable
 (c.e.)  sets under set inclusion.  Understanding the
 lattice-theoretic properties of $\mathcal{E}$ and the interplay
 between computability and definability in $\mathcal{E}$ are
 longstanding areas of research in classical computability theory.  In
 particular, researchers have worked to understand the automorphism
 group of $\mathcal{E}$ and the {\em orbits} of $\mathcal{E}$. The
 {\em orbit} of a c.e.\ set $A$ is the collection of c.e.\ sets
 \mbox{$[A]=\{B\in\mathcal{E}\mid (\exists\ \Psi:
   \mathcal{E}\xrightarrow{\sim}\mathcal{E})\,
   \left(\Psi(A)=B\right)\}$.}
 One of the major questions in classical computability is the
 following.

\begin{question}\label{Q:Orbits}
  What are the (definable) orbits of $\mathcal{E}$, and what degrees
  are realized in these orbits? How can new orbits be constructed from
  old ones? \end{question}

In seminal work \cite{Soare:74}, Soare proved that the maximal sets
form an orbit using his Extension Theorem. Martin \cite{M66b} had
previously shown that the maximal sets are exactly those c.e.\ sets of
high degree, thus describing the definable property of being maximal
in degree-theoretic terms.  In addition, Harrington had shown that the
creative sets form an orbit (see \cite{Soare:87}, Chapter XV).  In
time, Soare's Extension Theorem was generalized and applied widely to
construct many more orbits of $\mathcal{E}$. For example, Downey and
Stob \cite{DSAutsOrbits} showed that the {\em hemimaximal sets}, i.e.,
splits of maximal sets, form an orbit and studied their degrees. In
particular, any maximal or hemimaximal set is automorphic to a
complete set. On the other hand, Harrington and Soare \cite{HS1991}
defined a first order nontrivial property $Q$ such that if $A$ is a
c.e.\ set and $Q(A)$ holds, then $A$ is not automorphic to a complete
set.  These results are the first partial answers to the following
question related to Question \ref{Q:Orbits}.

\begin{question} \label{Q:AuttoComp} Which orbits of $\mathcal{E}$
  contain complete sets? \end{question} \noindent

It turns out that until recently all known definable orbits of
$\mathcal{E}$, besides the orbit of creative sets, were orbits of {\em
  $\mathcal{D}$-hhsimple sets}, generalizations of hhsimple
(hyper-hypersimple) sets (see \cite{Onorbits}).  (We give extensive
background on all definitions and ideas mentioned here in
\S\ref{D-maxbackground}.)

The Slaman-Woodin Conjecture \cite{SlamanWoodinConjecture} asserts
that the set
\begin{equation*}
  \{\langle i, j\rangle\mid (\exists\ \Psi: \mathcal{E}\xrightarrow{\sim}\mathcal{E})\, [\Psi(W_i)=W_j]\}
\end{equation*} is $\Sigma_1^1$-complete.  The conjecture was based on  the belief that information  could be coded into the  orbits of hhsimple sets.
Cholak, Downey and Harrington proved a stronger version of the
Slaman-Woodin Conjecture.

 \begin{theorem}[Cholak, Downey
and Harrington  \cite{Onorbits}]\label{CDH:SWthm} There is a \ce set $A$ such that the
index set \mbox{$\{i\in\omega\mid W_i\cong A\}$} is
$\Sigma_1^1$-complete.
\end{theorem}
\noindent
In a surprising twist (again see \cite{Onorbits}), the sets $A$
witnessing Theorem \ref{CDH:SWthm} cannot be simple or hhsimple
(showing that the original idea behind the conjecture fails).  It is
still open, however, if the sets in Theorem \ref{CDH:SWthm} can be
$\mathcal{D}$-hhsimple.
Moreover, the behavior of hhsimple sets under automorphisms is now
completely understood. Specifically, two hhsimple sets are automorphic
if and only if they are $\Delta_3^0$-automorphic \cite[Theorem
1.3]{mr2004f:03077}. There is no similar characterization of when
$\mathcal{D}$-hhsimple sets are automorphic.


Here we consider $\mathcal{D}$-maximal sets, a special case of
$\mathcal{D}$-hhsimple sets but a generalization of maximal sets, to
gain further insight into Questions \ref{Q:Orbits} and
\ref{Q:AuttoComp}.  A c.e.\ set $A$ is {\em $\mathcal{D}$-maximal} if
for all $W$ there is a c.e.\ set $D$ disjoint from $A$ such that
$W \subseteq^* A \sqcup D$ or $W \cup (A \sqcup D) =^* \omega$.  We
can understand a given $\mathcal{D}$-maximal set $A$ in terms of the
collection $\mathcal{D}(A)$ of c.e.\ sets that are disjoint from $A$.

The goal of this paper is to provide a complete nontrivial
classification of the $\mathcal{D}$-maximal sets in terms of how
$\mathcal{D}(A)$ is generated.  In Theorem
\ref{sec:nice-generating-sets-1}, we describe ten types of ways
$\mathcal{D}(A)$ can be generated for any c.e.\ set $A$.  We then show
in Theorem
\ref{sec:main-result-1} 
that there is a complete and incomplete $\mathcal{D}$-maximal set of
each type. 
The first six types of $\mathcal{D}$-maximal sets were already well
understood (\cite{Soare:74}, \cite{DSAutsOrbits}, \cite{MR2366962},
\cite{SomeOrbits}).  Furthermore, Herrmann and Kummer \cite{MR1264963}
had constructed $\mathcal{D}$-maximal sets that were not of the first
six types (in particular, as splits of hhsimple and atomless
$r$-maximal sets).  We, however, show that there are four types of
examples of $\mathcal{D}$-maximal sets besides the first six and that
each of these types breaks up into infinitely many orbits.  Moreover,
we provide an overarching framework for understanding and constructing
these examples.  We discuss $\mathcal{D}$-maximal sets of the first
six types and the type that arises as a split of an $r$-maximal set in
\S\ref{sec:main-result}.  In \S \ref{sec:build-hhsimple-like}, we show
how the remaining three types are very similar to the construction of
splits of hhsimple sets.  For ease of reading, we discuss open
questions as they arise.  In particular, open questions can be found
in \S\ref{Q0} 
and \S\ref{sec:anything-know-about}.

\section{Background and definitions}\label{D-maxbackground}

All sets considered in this paper are computably enumerable (c.e.),
infinite, and coinfinite unless explicitly specified.  Let
$\mathcal{E}^*$ be the structure $\mathcal{E}$ modulo the ideal of
finite sets $\mathcal{F}$.  By Soare \cite{Soare:74}, it is equivalent
to work with $\mathcal{E}^*$ instead of $\mathcal{E}$ in the sense
that two sets $A$ and $B$ are in the same orbit in $\mathcal{E}$ if
and only if they are in the same orbit in $\mathcal{E}^*$.  Given a
c.e.\ set $A$, we define
\begin{eqnarray*}
  \mathcal{L}(A)= (\{W \cup A\mid W \text{ a c.e.\ set}\}, \subseteq)\text{ and }\\
  \mathcal{E}(A)= (\{W \cap A\mid W \text{ a c.e.\ set}\}, \subseteq).
\end{eqnarray*}
We let $\mathcal{L}^*(A)$ be the structure $\mathcal{L}(A)$ modulo
$\mathcal{F}$, and $\mathcal{E}^*(A)$ be the structure
$\mathcal{E}(A)$ modulo $\mathcal{F}$.  Recall that $A$ is {\em
  maximal} if for all $B\in \mathcal{L}^*(A)$, if $B\not=^*A$, then
$B=^*\omega$.

If we understand the orbit of $A$, we can sometimes understand the
orbits of splits of $A$.
\begin{definition}\begin{itemize}
  \item[(i)] We call $A_0 \sqcup A_1 = A$ a {\em splitting of $A$},
    and we call $A_0$ and $A_1$ {\em splits of $A$} or {\em halves of
      the splitting of $A$}.  We say that this splitting is {\em
      trivial} if either of $A_0$ or $A_1$ are computable.
    \noindent
  \item[(ii)] We call $A_0 \sqcup A_1=A$ a \emph{Friedberg splitting}
    of $A$ if the following property holds for any c.e.\ $W$: if $W-A$
    is not c.e.\ then neither of $W-A_i$ are c.e. as well.

    \noindent
  \item[(iii)] Given a property $P$ of c.e.\ sets, we say that a
    noncomputable c.e.\ set $A$ is {\em hemi-P} if there is a
    noncomputable c.e.\ set $B$ disjoint from $A$ such that
    $A\sqcup B$ satisfies $P$.
  \end{itemize}
\end{definition}
Note that if $P$ is a definable property in $\mathcal{E}$ or
$\mathcal{E}^*$, then hemi-$P$ is also definable there.

\subsection{$\mathcal{D}$-hhsimple and $\mathcal{D}$-maximal sets}

\subsubsection{Motivation}

Recall that a coinfinite set $A$ is {hhsimple} if and only if
$\mathcal{L}^*(A)$ is a boolean algebra (\cite{Lachlan:68*3}, see also
\citet{Soare:87}). Hence, $A$ is maximal if and only if
$\mathcal{L}^*(A)$ is the two element boolean algebra.

\begin{theorem}[Lachlan \cite{Lachlan:68*3}]\label{T:Lachlan} If  a set $H$ is hhsimple, then
  $\mathcal{L}^*(H)$ is a $\Sigma^0_3$ boolean algebra.  Moreover, for
  every $\Sigma^0_3$ boolean algebra $\mathcal{B}$, there is a
  hhsimple set $H$ such that $\mathcal{L}^*(H)$ is isomorphic to
  $\mathcal{B}$.
\end{theorem}
\noindent
Given Theorem \ref{T:Lachlan}, we say that a hhsimple set $H$ has {\em
  flavor} $\mathcal{B}$ if $\mathcal{L}^*(H)$ is isomorphic to the
$\Sigma^0_3$ boolean algebra $\mathcal{B}$.  Note that the ordering
$\le$ on a $\Sigma_3^0$ boolean algebra is $\mathbf{0'''}$-computable.

\subsubsection{Working modulo $\mathcal{D}(A)$}

Given a set $A$, we define
\begin{equation*}
  \mathcal{D}(A) = \{B: B \in \mathcal{L}(A)\ \&\ B -
  A \mbox{ is c.e.}\},
\end{equation*}  and let  $\mathcal{D}^*(A)$ be
the structure $\mathcal{D}(A)$ modulo  $\mathcal{F}$.
Since $\mathcal{D}^*(A)$ is an ideal in the lattice
$\mathcal{L}^*(A)$, we can take the quotient lattice
$\mathcal{L}^*(A)/ \mathcal{D}^*(A)$.  Theorem \ref{T:Lachlan}
motivates the following definition.

\begin{definition} (Herrmann and Kummer
  \cite{MR1264963})\label{BAdefDhhsimple} A set $A$ is {\em
    $\mathcal{D}$-hhsimple} if $\mathcal{L}^*(A)/\mathcal{D}^*(A)$ is
  a boolean algebra, and $A$ is {\em $\mathcal{D}$-maximal} if
  $\mathcal{L}^*(A)/\mathcal{D}^*(A)$ is the two element boolean
  algebra.
\end{definition}
\noindent
By unraveling Definition \ref{BAdefDhhsimple}, we have the following
working definition of $\mathcal{D}$-maximality.

\begin{definition}
  \label{def:Dmaxequiv}
  A set $A$ is {\em $\mathcal{D}$-maximal} if for all $W$ there is a
  c.e.\ set $D$ disjoint from $A$ such that $W \subseteq^* A \sqcup D$
  or $W \cup (A \sqcup D) =^* \omega$.
\end{definition}

Another useful characterization of the $\mathcal{D}$-maximal sets is
given in the next lemma.

\begin{lemma}[\citet{SomeOrbits} Lemma 2.2]\label{L:Dmaxchar}
  Let $A$ be a c.e.\ noncomputable set.  The set $A$ is
  $\mathcal{D}$-maximal if and only if, for all c.e.\ $W\supseteq A$,
  either $W-A$ is c.e. or there exists a computable $R$ such that
  $A\subseteq R\subseteq W$. \end{lemma}

Herrmann and Kummer \cite{MR1264963} studied the
$\mathcal{D}$-hhsimple sets in the context of {\em diagonal} sets. A
set is {\em diagonal} if it has the form
$\{e\in\omega\mid \psi_e(e)\}$ for some computable enumeration
$\{\psi_i\}_{i\in\omega}$ of all partial computable functions. In
\cite{MR1264963}, they showed that a set is not diagonal if and only
if it is computable or $\mathcal{D}$-hhsimple. Note that this result
implies that the property of being diagonal is elementary
lattice-theoretic.

\subsection{Known examples of $\mathcal{D}$-maximal
  sets}\label{S:introexDmax}

Maximal sets and hemimaximal sets (which form distinct orbits
\cite{Soare:74}, \cite{DSAutsOrbits}) are clearly
$\mathcal{D}$-maximal.  Similarly, a set that is maximal on a
computable set is also $\mathcal{D}$-maximal.
In these three cases, for the $\mathcal{D}$-maximal set $A$ there is a
$W$ such $A \cup W$ is maximal.  As we will see, this does not occur
for other types of $\mathcal{D}$-maximal sets.  
Others, however, have constructed additional kinds of
$\mathcal{D}$-maximal sets, in particular, {\em Herrmann} and {\em
  hemi-Herrmann sets} and sets with {\em $A$-special lists}, which we
define now.  It is easy to check that these sets are
$\mathcal{D}$-maximal from their respective definitions.

\begin{definition}
  \begin{itemize}
  \item[(i)] We say that a c.e.\ set $A$ is {\em $r$-separable} if,
    for all c.e.\ sets $B$ disjoint from $A$, there is a computable
    set $C$ such that $B\subseteq C$ and $A\subseteq \overline{C}$.
    We say that $A$ is {\em strongly} $r$-separable if, additionally,
    we can choose $C$ so that $C-B$ is infinite.
  \item[(ii)] We say that a set $A$ is {\em Herrmann} if $A$ is both
    $\mathcal{D}$-maximal and strongly $r$-separable.

  \item[(iii)] Given a set $A$, we call a list of c.e.\ sets
    $\mathcal{F} = \{ F_i : i \in \omega\}$ an \emph{$A$-special list}
    if $\mathcal{F}$ is a collection of pairwise disjoint
    noncomputable sets such that $F_0 = A$ and for all c.e.\ sets $W$,
    there is an $i$ such that $W \subseteq^* \bigsqcup_{l\leq i}F_l$
    or $W \cup \bigsqcup_{l\leq i}F_l =^* \omega$.

  \item[(iv)] We say a set $A$ is \emph{$r$-maximal} if for every
    computable set $R$, either $R \cap \overline{A} = ^* \emptyset$
    (so $R \subseteq^* A$) or
    $\overline{R} \cap \overline{A} =^* \emptyset$ (so
    $\overline{A} \subseteq^* R$), i.e., no infinite computable set
    splits $\overline{A}$ into two infinite sets.

  \item[(v)] A c.e.\ set $B$ is \emph{atomless} if for every c.e.\ set
    $C$, if $B\subseteq C\neq^* \omega$, then there is a c.e.\ set $E$
    such that $C \subsetneq^* E \subsetneq^* \omega$, i.e., $B$ does
    not have a maximal superset.

  \end{itemize}
\end{definition}

Herrmann and hemi-Herrmann sets were defined by Hermann and further
discussed in \cite{SomeOrbits}.  The main results in \cite{SomeOrbits}
for our purposes are that such sets exist (Theorem 2.5) and that these
sets form distinct (Theorem 6.9) definable (Definition 2.3)
orbits (Theorems 4.1, 6.5) each containing a complete set (Theorems
7.2, 6.7(i)).

The notion of a set $A$ with an $A$-special list was introduced in
\cite[\S 7.1]{MR2366962}. There, Cholak and Harrington showed that
such sets exist and form a definable $\Delta^0_4$ but not $\Delta^0_3$
orbit.  This orbit remains the only concrete example of an orbit that
is not $\Delta^0_3$.
Furthermore, as mentioned earlier Herrmann and Kummer \cite{MR1264963}
had constructed $\mathcal{D}$-maximal splits of hhsimple and atomless
$r$-maximal sets in addition to the ones mentioned above.  We will
discuss these examples later (see \S\ref{sec:herrm-kumm-result}), but
first we explore the notion of a {\em generating set} for
$\mathcal{D}(A)$ for an arbitrary (not necessarily
$\mathcal{D}$-maximal) set $A$.

\section{Generating sets for
  $\mathcal{D}(A)$}\label{sec:results-about-sets}

In this section, we only assume that the sets considered are
computably enumerable. In later sections, we will work explicitly with
$\mathcal{D}$-maximal sets.  We will use the framework of {\em
  generating sets} to understand and classify the different kinds of
$\mathcal{D}$-maximal sets.

\begin{definition}
  We say a (possibly finite or empty) collection of c.e.\ sets
  $\mathcal{G} = \{D_0, D_1, \ldots \}$ {\em generates}
  $\mathcal{D}(A)$ (equivalently $\mathcal{G}$ is a {\em{generating
      set} for $\mathcal{D}(A)$}) if each $D_i$ is disjoint from $A$
  for all $i\in\omega$ and for all c.e.\ sets $D$ that are disjoint
  from $A$, there is a finite set $F\subset \omega$ such that
  $D \subseteq^* \bigcup_{j\in F}D_j$.  In this case, we say that
  $\{D_j\mid j\in F\}$ {\em covers} $D$.  If $\mathcal{G}$ generates
  $\mathcal{D}(A)$, we write $\mathcal{D}(A)=\gen{\mathcal{G}}$.  We
  say $\{ D_0, D_1, \ldots \}$ {\em partially generates}
  $\mathcal{D}(A)$ if there is some collection of sets $\mathcal{G}$
  containing $\{ D_0, D_1, \ldots \}$ such that
  $\gen{\mathcal{G}}=\mathcal{D}(A)$.
\end{definition}

\noindent
We list a few basic observations.

\begin{lemma}\label{sec:gener-sets-simple}
{\em(i)} Generating sets always exist for $\mathcal{D}(A)$. In particular,
$\mathcal{D}(A)$ is generated by the collection of all c.e.\ sets
that are disjoint from $A$.  \\
\noindent {\em (ii)} Let $\Phi$ be an automorphism of
$\mathcal{E}^*$. If for all c.e.\ $W$, we set $\widehat{W}:=\Phi(W)$,
then $\{D_0, D_1, \ldots\}$ generates $\mathcal{D}(A)$ if and only if
$\{\widehat{D}_0, \widehat{D}_1, \ldots
\ldots \}$ generates $\mathcal{D}(\widehat{A})$.\\
\noindent {\em (iii)} $\mathcal{D}(A) = \gen{\emptyset}$ iff $A$ is
simple.
\end{lemma}

\subsection{Simplifying generating sets}\label{SS:simplegensets}
Generating sets for $\mathcal{D}(A)$ are far from unique.
Here we develop some tools for finding less complex generating sets
for $\mathcal{D}(A)$.  We use different tools based on whether or not
$\mathcal{D}(A)$ has a finite generating set.

\subsubsection{Finite generating sets}

\begin{lemma}\label{sec:finitegen}
  If a finite collection of sets ${\mathcal{G}}$ generates
  $\mathcal{D}(A)$, then $\mathcal{D}(A)=\gen{\emptyset}$,
  $\mathcal{D}(A)= \gen{R}$ for some infinite computable set $R$, or
  $\mathcal{D}(A)= \gen{W}$ for some noncomputable c.e. set $W$.
  Moreover, if $\{ D \}$ and $\{ \tilde{D} \}$ both generate
  $\mathcal{D}(A)$, then $D =^* \tilde{D}$. \end{lemma}

\begin{proof}
  The union $W$ of the finitely many sets in ${\mathcal{G}}$ is
  c.e. and disjoint from $A$ and clearly generates
  $\mathcal{D}(A)$. If $W$ is finite, then
  $\mathcal{D}(A)= \gen {\emptyset }$, and otherwise, we are in the
  remaining two cases.  For the last statement,
  $\tilde{D} \subseteq^* D$ and $D \subseteq^* \tilde{D}$ by the
  definition of generating set.
\end{proof}

The collection of all c.e.\ sets that have finite generating sets is
definable.
\begin{lemma}\label{L:FinGenDef}
  The statement ``A single set generates $\mathcal{D}(A)$" is an
  elementarily definable statement in $\mathcal{E}^*$ under
  inclusion. \end{lemma}

\subsubsection{Infinite generating sets} Infinite generating sets can
be much more complex, depending on whether all (or many of) the
elements can be chosen to be computable or pairwise disjoint.

\begin{lemma}\label{sec:computable}
  If $\{{R}_0, {R}_1, \ldots \}\cup \mathcal{G}$ generates
  $\mathcal{D}(A)$ where ${R}_i$ is computable for all $i\in\omega$,
  then there exists a pairwise disjoint collection of computable sets
  $\{\tilde{R}_0, \tilde{R}_1, \ldots \}$ so that
  $\{\tilde{R}_0, \tilde{R}_1, \ldots \}\cup \mathcal{G} $ generates
  $\mathcal{D}(A)$.
\end{lemma}

\begin{proof}
  If $\{{R}_0, {R}_1, \ldots \}\cup \mathcal{G}$ generates
  $\mathcal{D}(A)$, we inductively
  define 
  $\{\tilde{R}_0, \tilde{R}_1, \ldots \}$. Let $\tilde{R}_0 =
  {R}_0$.
  Given the pairwise disjoint collection of computable sets
  $\{\tilde{R}_0, \ldots, \tilde{R}_n \}$, let $m$ be the least index
  such that ${R}_m - \bigsqcup_{i\le n} \tilde{R}_i$ is infinite. If
  no such $m$ exists,
  $\{\tilde{R}_0, \ldots, \tilde{R}_n \}\cup \mathcal{G}$ generates
  $\mathcal{D}(A)$. 
  
  Otherwise, let
  $\tilde{R}_{n+1} = {R}_m - \bigsqcup_{i\le n} \tilde{R}_i = {R}_m
  \cap \overline{\bigsqcup_{i\le n} \tilde{R}_i}$.
  The collection $\{\tilde{R}_0, \tilde{R}_1, \ldots \}$ satisfies the
  conclusion of the lemma since for each $i\in\omega$ there is an $m$
  such that
  $\bigcup_{j\leq i} {R}_j \subseteq^* \bigsqcup_{j\leq m}
  \tilde{R}_j$.
\end{proof}



\begin{lemma}\label{sec:disjoint}
  If $\mathcal{D}(A)$ is generated by an infinite collection of
  pairwise disjoint sets, then $\mathcal{D}(A)$ is also generated by
  an infinite collection of pairwise disjoint sets containing only
  computable sets, only noncomputable sets, or only computable sets
  and one noncomputable set.
\end{lemma}

\begin{proof}
  Suppose that the collection
  $\mathcal{G}=\{D_0, D_1, \ldots, R_0, R_1, \ldots \}$ generates
  $\mathcal{D}(A)$ and consists of pairwise disjoint sets such that
  $D_i$ is noncomputable and $R_i$ is computable for all $i\in\omega$.
  Suppose that there are both computable and noncomputable sets in
  $\mathcal{G}$.  We may assume all of these sets are infinite.  If
  $\mathcal{G}$ contains only finitely many $D_i$, the finite union of
  the $D_i$s together with $\{R_0, R_1, \ldots \}$ generates
  $\mathcal{D}(A)$.  If $\mathcal{G}$ contains infinitely many $D_i$s,
  then $\{\tilde{D}_0, \tilde{D}_1, \ldots\}$, where
  $\tilde{D}_i := D_i \sqcup R_i$, generates $\mathcal{D}(A)$. Note
  that $D_i$ being noncomputable implies that $\tilde{D}_i$ is
  noncomputable.  \end{proof}

\begin{lemma}\label{sec:computable2}
  If $\{ R_0, R_1, \ldots \}$ and $\{ D_0, D_1, \ldots \}$ are
  pairwise disjoint generating sets for $\mathcal{D}(A)$ and all sets
  in $\{ R_0, R_1, \ldots \}$ are computable, then all sets in
  $\{ D_0, D_1, \ldots \}$ are computable.
\end{lemma}

\begin{proof}
  By definition of generating sets, there is a finite
  $F\subset \omega$ such that
  $D_i \subseteq^* \bigsqcup_{j \in F} R_j$.  It suffices to show that
  $\bigsqcup_{j \in F} R_j - D_i$ is a c.e.\ set.  There is a finite
  $H\subset\omega$ such that
  $\bigsqcup_{j \in F} R_j \subseteq^* \bigsqcup_{j \in H} D_j
  $.
  Since the $D_i$ are pairwise disjoint, $i \in H$. Set
  $\tilde{H}:= H - \{i \}$. Then
  $(\bigsqcup_{j \in F} R_j -D_i) \subseteq^* \bigsqcup_{j \in
    \tilde{H}} D_j $.
  Since $D_i$ and $ \bigsqcup_{j \in \tilde{H}} D_j $ are disjoint,
  \begin{equation*}\bigsqcup_{j \in
      F} R_j - D_i =^* \bigsqcup_{j \in F} R_j \cap \bigsqcup_{j \in
      \tilde{H}} D_j, \text{ which is c.e.}
  \end{equation*}
\end{proof}

We may assume that we have a generating set for $\mathcal{D}(A)$ whose
union is $\bar{A}$.

\begin{lemma}\label{sec:all}
  If $\{D_0, D_1, \ldots \}\cup\mathcal{G}$ generates $\mathcal{D}(A)$
  (and $\{D_0, D_1, \ldots \}$ is a collection of pairwise disjoint
  sets), then there exists a (pairwise disjoint) collection of sets
  $\{\tilde{D}_0, \tilde{D}_1, \ldots \}$ such that
  $\overline{A}= \bigsqcup_{i \in \omega} \tilde{D}_i$ and
  $\{\tilde{D}_0, \tilde{D}_1, \ldots \}\cup\mathcal{G}$ generates
  $\mathcal{D}(A)$.  \end{lemma}

\begin{proof}
  If $X = \overline{A} - \bigsqcup_{i \in \omega} D_i$ and
  $X = \{ x_0 < x_ 1< \ldots\}$, we can take
  $\tilde{D}_i:=D_i \sqcup \{x_i\}$.
\end{proof}

We can also simplify partial generating sets that are not pairwise
disjoint.

\begin{lemma}\label{sec:nested}
  Let $\{ D_0, D_1, \ldots \}$ be a list of noncomputable c.e.\ sets
  whose union with $\mathcal{G}$ generates $\mathcal{D}(A)$.  Then,
  there is a collection of noncomputable c.e.\ sets
  $\{ \tilde{D}_0, \tilde{D}_1, \ldots \}$ whose union with
  $\mathcal{G}$ generates $\mathcal{D}(A)$ such that all the sets are
  either pairwise disjoint or nested so that
  $\tilde{D}_{n+1}-\tilde{D}_n$ is not c.e.\ for all
  $n\in\omega$. \end{lemma}

\begin{proof}
  In a highly noneffective way, we build a list
  $\{ \tilde{D}_0, \tilde{D}_1, \ldots \}$, satisfying our
  conclusion. To ensure that this list partially generates
  $\mathcal{D}(A)$ as described, we construct this list so that each
  $\tilde{D}_i$ is disjoint from $A$ and every $D_i$ is contained in
  the union of finitely many $\tilde{D}_i$'s.

  We attempt to inductively construct the list to consist of pairwise
  disjoint sets based on an arbitrary starting point $k\in \omega$.
  For each $k\in\omega$, we inductively define a function
  $l_k:\omega\to\omega$.  We set $l_k(0):=k$ and
  $\tilde{D}_0 = \bigcup_{i \leq l_k(0)} {D}_i$. We let $l_k(n+1)$ be
  the least number (if it exists) such that
  $\bigcup_{i \leq l_k(n+1) } D_i -\bigcup_{i\leq l_k(n)} D_i$ is a
  c.e.\ set.  Let $\tilde{D}_{n+1}$ be this c.e.\ set. Then,
  $\bigsqcup_{i \leq n} \tilde{D}_i = \bigcup_{i \leq l_k(n)} {D}_i$.
  If for some initial choice of $k$, the function $l_k$ has domain
  $\omega$, then the sets in $\{ \tilde{D}_0, \tilde{D}_1, \ldots \}$
  are pairwise disjoint.

  Otherwise, the above procedure fails for all initial choices of
  $k$. Then, each $l_k$ is a strictly increasing function defined on
  some nonempty finite initial segment of $\omega$. Let
  $m:\omega\to\omega$ be defined so that $m(k)$ is the maximum value
  of $l_k$. For all $k$, $m(k) \geq k$. Moreover, for all $k$ and
  $l>m(k)$, $\bigcup_{ i \leq l } D_i -\bigcup_{i \leq m(k)} D_i$ is
  never a c.e.\ set. We define a strictly increasing function
  $\tilde{m}:\omega\to\omega$ inductively by setting
  $\tilde{m}(0)= m(0)$ and $\tilde{m}(n+1) = m(\tilde{m}(n)+1)$. By
  construction, the list given by
  $\tilde{D}_{n} = \bigcup_{ i \leq \tilde{m}(n) } D_i$ has the
  desired nesting property. \end{proof}

\subsection{Standardized Types of generating
  sets}\label{SS:TypesgensetsThm}

We use the results from \S\ref{SS:simplegensets} to show
that any c.e.\ set $A$ has a generating set for $\mathcal{D}(A)$ of
one of ten standardized types.  We can then classify c.e.\ sets by the
complexity of their generating sets (see
Definition~\ref{sec:typedefn}).

\begin{theorem}\label{sec:nice-generating-sets-1}
  For any c.e.\ set $A$, there exists a collection of c.e.\ sets,
  $\mathcal{G}$, generating $\mathcal{D}(A)$ of one of the following
  types.

  \begin{description}
  \item[Type 1] $\mathcal{G} = \{ \emptyset \}$.
  \item[Type 2] $\mathcal{G} = \{ R \}$, where $R$ is an infinite
    computable set.
  \item[Type 3] $\mathcal{G} = \{ W \}$, where $W$ is an infinite
    noncomputable set.
  \item[Type 4] $\mathcal{G} = \{ R_0, R_1, \ldots \}$, where the
    $R_i$ are infinite pairwise disjoint computable sets.
  \item[Type 5] $\mathcal{G} = \{ D_0, R_0, R_1, \ldots \}$, where
    $D_0$ is the only noncomputable set and all the sets are infinite
    and pairwise disjoint. 
  \item[Type 6] $\mathcal{G} = \{ D_0, D_1, \ldots \}$, where the
    $D_i$ are infinite pairwise disjoint noncomputable sets.
  \item[Type 7] $\mathcal{G} = \{ D_0 , R_0, R_1, \ldots \}$, where
    $D_0$ is the only noncomputable set, the $R_i$ are infinite
    pairwise disjoint computable sets, and
    $D_0 \cap R_i \neq \emptyset$ for infinitely many $i$. 
  \item[Type 8]
    $\mathcal{G} = \{ D_0, D_1, \ldots , R_0, R_1, \ldots \}$, where
    the $D_i$ are pairwise disjoint noncomputable sets and the $R_i$
    are infinite pairwise disjoint computable sets.
  \item[Type 9]
    $\mathcal{G} = \{ D_0, D_1, \ldots , R_0, R_1, \ldots \}$, where
    the $R_i$ are infinite pairwise disjoint computable sets and the
    $D_i$ are infinite nested noncomputable sets such that, for all
    $l\in\omega$, $D_{l+1} - D_l$ is not c.e.\ and there are
    infinitely many $j$ such that $R_j - D_l$ is infinite.
  \item[Type 10] $\mathcal{G} = \{ D_0, D_1, \ldots \}$, where the
    $D_i$ are infinite nested noncomputable sets such that
    $D_{l+1} - D_l$ is not c.e.\ for all $l\in\omega$.
  \end{description}
\end{theorem}

\begin{proof}
  (1) If there is a finite generating set for $\mathcal{D}(A)$, then
  there is a generating set of Type 1, 2 or 3 for $\mathcal{D}(A)$ by
  Lemma~\ref{sec:finitegen}.

  (2) If $\mathcal{D}(A)$ has an infinite generating set consisting of
  pairwise disjoint sets, then $\mathcal{D}(A)$ has a generating set
  of Type 4, 5, or 6 by Lemma~\ref{sec:disjoint}.  We remark that, by
  Lemma \ref{sec:computable2}, if $\mathcal{D}(A)$ has a generating
  set of Type 4, then $\mathcal{D}(A)$ does not have a generating set
  of Type 5 or Type 6. Note that if $\mathcal{D}(A)$ has an infinite
  generating set consisting only of computable sets, we can assume
  these computable sets are pairwise disjoint by
  Lemma~\ref{sec:computable} and hence (2) holds.

  Assume the antecedents of (1) and (2) fail, and take some generating set for $\mathcal{D}(A)$.  By repeatedly taking finite unions of some
  of the sets, we can assume that there are zero,
  one, or infinitely  many computable sets  in this generating set.  If there is one computable
  set $R$, then we can assume that all other sets are noncomputable and
  disjoint from $R$ (by  removing $R$ from each of those sets). 
  We can then take one of the noncomputable sets, $W$, and replace it
  with $R \sqcup W$.  So, we can assume that the generating set has either no computable
  sets or infinitely many pairwise disjoint computable sets (again by   Lemma~\ref{sec:computable}).  Similarly, by
  taking finite unions, we can assume the generating set also has  zero, one, or infinitely many
  noncomputable sets. By the failure of the antecedent of (2)  and Lemma~\ref{sec:computable},
  having zero is not a possibility.

  (3) If there is one noncomputable set, then
  $\mathcal{D}(A)$ has a generating set of Type 7 since  the antecedents of (1) and (2) fail.   Specifically, there must be
  infinitely many disjoint computable sets in the generating set.  If only finitely many of these computable sets intersected with the one noncomputable set $W$, we could replace $W$ by its union with these finitely many  sets to obtain a Type 5 generating set, a contradiction.  

  If the antecedent of (3) fails, the generating set contains infinitely many noncomputable sets
  $\{D_0, D_1, \ldots \}$, and, by Lemma~\ref{sec:nested},
  these noncomputable sets can be taken to be either pairwise disjoint
  or nested so that $D_n\subset D_{n+1}$ and ${D}_{n+1}-{D}_n$ is not
  c.e.\ for all $n\in\omega$. If the generating set contains infinitely  many computable sets and the infinitely
  many noncomputable sets are pairwise disjoint, this generating set is of Type 8.

  (4) If $\mathcal{D}(A)$ has a generating set 
  of Type 8 or Type 9, we are done.    
  
  Now we will argue that the failure of all the antecedents of (1) though (5)
  implies  that $\mathcal{D}(A)$ has a generating set of Type 10.  Notice
  that if no computable sets remain in our generating set, we are done.  So, assume
  otherwise. Hence, the noncomputable sets are nested, and almost all the
  computable sets are almost contained in one of the noncomputable
  sets.  For each noncomputable set $D_i$ in this generating set, we
  can take the union of $D_i$ and the remaining finitely many
  computable sets to obtain a new generating set
  $\{\tilde{D}_0, \tilde{D}_1, \ldots \}$ where the $\tilde{D}_i$ are
  infinite nested sets.  This generating set cannot contain infinitely many computable
  $\tilde{D}_i$ since the collection of the computable $\tilde{D}_i$
  would be a generating set consisting of only computable sets by the note after (2).  Hence, there are only
  finitely many computable $\tilde{D}_i$.  The collection of
  noncomputable $\tilde{D}_i$ also generates $\mathcal{D}(A)$ since
  the $\tilde{D}_i$ are nested.  If this generating set is not of Type
  10, we can apply Lemma \ref{sec:nested} to obtain one of Type 10
  since the antecedent of (2) fails.
\end{proof}

Note that $\mathcal{D}(A)$ may have generating sets of different
Types.
However, the Types are listed in order of increasing complexity.  By
following the procedure outlined in the proof of Theorem
\ref{sec:nice-generating-sets-1}, we will always find a generating set
for $\mathcal{D}(A)$ of lowest possible complexity.  Hence, we can
classify the c.e.\ sets by the Type complexity of their generating
set.  

\begin{definition}\label{sec:typedefn}
  We say the c.e.\ set ${A}$ is {\em Type $n$} if there is a
  generating set for $\mathcal{D}(A)$ of Type $n$ but no generating
  set for $\mathcal{D}(A)$ of Type $m$ for all $m < n$.
\end{definition}
\noindent
In Theorem \ref{sec:main-result-1}, the main result of this paper, we show there is a  $\mathcal{D}$-maximal set of each Type.

We more closely examine sets of a given Type in \S
\ref{SS:underTypes}, but it is helpful to first observe the behavior
of generating sets under splitting.

\subsection{Splits and generating sets for $\mathcal{D}(A)$}

\begin{lemma}\label{sec:splits2}
  Suppose $R$ is computable, $A_0 \sqcup R = A$, and
  $\mathcal{G} \subseteq \mathcal{D}(A)$. Then $\mathcal{G}$ generates
  $\mathcal{D}(A)$ iff $\mathcal{G} \cup \{R\}$ generates
  $\mathcal{D}(A_0)$. \end{lemma}

\begin{proof}
  Let $D$ be c.e.\ and disjoint from $A_0$. Since
  $D-A = D \cap \overline{R}$ is c.e.\ and disjoint from $A$, the set
  $D$ is covered by $R$ and finitely many sets in
  $\mathcal{G}$. \end{proof}

\begin{corollary}\label{sec:splits2cor}
  A set $A$ is half of a trivial splitting of a simple set iff
  $\mathcal{D}(A) = \gen{R}$ for some computable $R$. \end{corollary}

\begin{lemma} Suppose $A_0 \sqcup A_1=A$.  If
  $\mathcal{G} \cup \{A_1\}$ generates $\mathcal{D}(A_0)$, then
  $\mathcal{G} $ generates $\mathcal{D}(A)$.
\end{lemma}

\begin{lemma}\label{sec:splits1} Suppose $\mathcal{G}$ generates
  $\mathcal{D}(A)$ and $A_0 \sqcup A_1=A$.  If $A_0 \sqcup A_1$ is a
  Friedberg splitting of $A$, then $\mathcal{G} \cup \{A_1\}$
  generates $\mathcal{D}(A_0)$. \end{lemma}

\begin{proof}
  Let $D$ be a c.e.\ set. If $D-A$ is not c.e., then $D-A_0$ is not
  c.e. and $D$ is not disjoint from $A_0$. So, assume that $D-A$ is a
  c.e.\ set. If $D$ is disjoint from $A_0$, then $D$ is covered by
  $A_1$ and finitely many sets in $\mathcal{G}$.
\end{proof}

\begin{corollary}\label{sec:splits1cor}
  If $A$ is half of a Friedberg splitting of a simple set, then
  $\mathcal{D}(A) = \gen{W}$ where $W$ is not
  computable. \end{corollary}

Suppose that $A_0 \sqcup A_1=A$ is a nontrivial splitting that is not
Friedberg.  We would like a result describing a generating set for
$\mathcal{D}(A_0)$ similar to Lemmas \ref{sec:splits2} and
\ref{sec:splits1}, but such a result is not clear.  For the splitting
$A_0 \sqcup A_1=A$, there is a set $W$ such that $W-A$ is not c.e.\
but $W-A_0$ is a c.e.\ set. Since $W-A$ may not be contained in a
finite union of generators for $\mathcal{D}(A)$ (for example, if $A$
is simple), the set $A_1$ and the generators for $\mathcal{D}(A)$ may
not generate $\mathcal{D}(A_0)$.  Also,
the work in Section~\ref{sec:type-10-sets} shows that the converse of
Lemma \ref{sec:splits1} fails; for a splitting $A_0\sqcup A_1$, the
collection $\mathcal{G}\cup \{A_1\}$ generating $\mathcal{D}(A_0)$
does not mean that the splitting is Friedberg.

\subsection{Understanding the Types}\label{SS:underTypes}
Sets of Types 1, 2, and 3 are particularly well understood.  By
Lemma~\ref{sec:gener-sets-simple}, $S$ is simple iff $S$ is of Type 1;
there are no infinite c.e.\ sets disjoint from $S$. By this fact and
Lemma~\ref{sec:splits2}, $A \sqcup R$ is simple iff $A$ is Type 2. By
Lemma~\ref{sec:splits1}, if $A$ is half of a Friedberg splitting of a
simple set, then $A$ is of Type 3.  Moreover, sets of these Types are
definable.

\begin{lemma}\label{sec:nice-generating-sets}
  The statement ``$A$ is Type 1 (respectively 2, 3)'' is elementarily
  definable in $\mathcal{E}^*$ under inclusion.
\end{lemma}

\begin{proof}
  The set $A$ is Type 1 iff $A$ is simple, and $A$ is Type 2 iff there
  is a computable set $R$ disjoint from $A$ such that $A \sqcup R$ is
  simple. The set $A$ is Type 3 iff there a c.e. set $D$ such that $D$
  is disjoint from $A$ and for all c.e.\ sets $W$ disjoint from $A$,
  $W \subseteq^*
  D$.  
\end{proof}

In \S\ref{sec:main-result}, we will show that there are
$\mathcal{D}$-maximal sets of all ten Types.  Moreover, we will show
that $\mathcal{D}$-maximal sets of Type 4, 5 and 6 are definable,
somewhat extending Lemma \ref{sec:nice-generating-sets}.  However, the
following question is open.
\begin{question}\label{Q:alltypesdefinable}
  Is there a result similar to Lemma \ref{sec:nice-generating-sets}
  for the remaining Types of sets in a general setting?
\end{question}

We finish this section with a remark on the behavior of sets of
various Types under trivial or Friedberg splitting.

\begin{remark}\label{sec:splits-gener-mathc}
  Suppose $A_0 \sqcup A_1 = A$ is a trivial or Friedberg splitting and
  $A_0$ is not computable. If $A$ is Type 1 or 2, then either $A_0$ is
  Type 2 ($A_0 \sqcup A_1 = A$ is a trivial splitting) or Type 3
  ($A_0 \sqcup A_1 = A$ is a Friedberg splitting).  If $A$ is Type 3,
  then $A_0$ is Type 3. If $A$ is Type 4, then $A_0$ is Type 4
  ($A_0 \sqcup A_1 = A$ is a trivial splitting) or Type 5. If $A$ is
  Type 5 (6, 7, or 8), then $A_0$ is Type 5 (6, 7, or 8) (replace
  $D_0$ with the union of $D_0$ and $A_1$). If $A$ is Type 9 (10),
  then $A_0$ is Type 9 (10) (replace each $D_i$ with the union of
  $D_i$ and $A_1$). \end{remark}

We now examine the last four types more carefully.  First, we explore
the subtle difference between Types 9 and 10, which is encoded in the
last clauses of these Types' definitions.

\subsubsection{Type 10 sets and
  $r$-maximality}\label{sec:generating-sets-r}

Type 10 sets can arise as splits of $r$-maximal sets.

\begin{lemma}\label{sec:notsplitrmaxi}
  If $A$ is half of a splitting of an $r$-maximal set (so not of Type
  1) and $A$ is not Type 2 or 3, then $A$ is Type 10. \end{lemma}

\begin{proof}
  We will show that if $A$ is Type 4, 5, 6, 7, 8, or 9 (and hence not
  of Type 1, 2, 3, or 10) then $A$ is not half of a splitting of an
  $r$-maximal set. Fix some infinite generating set $\mathcal{G}$ for
  $\mathcal{D}(A)$ of Type 4, 5, 6, 7, 8, or 9.

  Let $B$ be a c.e.\ set disjoint from $A$ (such sets exist since $A$
  is not Type 1). We show that $A\sqcup B$ is not $r$-maximal.  Since
  $\mathcal{G}$ is a generating set, $B$ is contained in some finite
  union of sets in $\mathcal{G}$. Every c.e.\ superset of an
  $r$-maximal set is either almost equal to $\omega$ or $r$-maximal
  itself. Since $A$ does not have Type 2 or 3,
  $A\sqcup B\not=^*\omega$ and we can assume $B$ is the union of these
  finitely many generators.  We proceed by cases. For $\mathcal{G}$ of
  Type 4, 5 or 7, an $R_i$ not part of the union witnesses that
  $A\sqcup B$ is not $r$-maximal.  For $\mathcal{G}$ of Type 6, an
  infinite computable subset of some $D_i$ not part of the union
  demonstrates that $A\sqcup B$ is not $r$-maximal.  For $\mathcal{G}$
  of Type 8 or 9, assume that
  $B \subseteq^* \bigcup_{j \leq i} R_j \cup \bigcup_{j \leq i} D_j$.
  If $\mathcal{G}$ has Type 8, there is some $l > i$ such that
  $D_l\cap \overline{ \bigcup_{j \leq i} R_j }$ is infinite.  An
  infinite computable subset of this intersection demonstrates that
  $A\sqcup B$ is not $r$-maximal.  Finally, suppose $\mathcal{G}$ is
  Type 9. By the last clause of Type 9, there is an ${r} >i $ such
  that $R_{r} - D_i$ is infinite. The computable set $R_{r}$ witnesses
  that $A\sqcup B$ is not $r$-maximal.

\end{proof}

Note that we cannot eliminate the assumption that $A$ is not Type 2 or
3 in Lemma \ref{sec:notsplitrmaxi}. If $A \sqcup R$ is a trivial
splitting of an $r$-maximal set, then $A\sqcup R$ is simple.  By
Corollary~\ref{sec:splits2cor}, $\mathcal{D}(A) = \{ R\}$ and $A$ is
Type 2. Similarly, by Corollary~\ref{sec:splits1cor}, if $A \sqcup B$
is a Friedberg splitting of an $r$-maximal set, $A$ is Type 3.

\begin{question}\label{Q:converseType10}
    If $A$ is $\mathcal{D}$-maximal and Type 10, then is $A$ half of a
  splitting of an  $r$-maximal set?
\end{question}

 We can, however, prove a stronger version of this statement with an additional assumption.

\begin{lemma}\label{sec:type-10-dmax}
  If $A$ is $\mathcal{D}$-maximal and Type 10, then $A$ is half of a
  splitting of an atomless $r$-maximal set. \end{lemma}

\begin{proof}   Let $\mathcal{G}=\{ D_0, D_1, \ldots \}$ be a Type 10
generating set for $\mathcal{D}(A)$.  
We first show that  $A \sqcup D_i$ is $r$-maximal
for some $i\in\omega$.  Suppose otherwise.  We construct an infinite collection $\left\lbrace R_0, R_1, \ldots\right\rbrace$ of infinite computable
pairwise disjoint sets all disjoint from $A$ such that $R_i - D_i$ is
infinite for all $i\in\omega$.  As $D_{i+1} \supseteq D_i$ this entails that $R_i - D_l$ is infinite for all $i>l$. Thus $G' = \left\lbrace R_0, R_1, \ldots, D_0, D_1, \ldots \right\rbrace$ is a generating set for $\mathcal{D}(A)$ of Type 9 and by Definition
\ref{sec:typedefn}, $A$ has Type at most $9$.

We assume inductively that  $R_0, \ldots, R_n$ are infinite computable
pairwise disjoint sets all disjoint from $A$ such that $R_j - D_j$ is
infinite for $j\leq n$.  Suppose that $A \cup D_{n+1} \cup\bigsqcup_{i
\leq n} R_n = C$ is not $r$-maximal. So, $\overline{C}$ is split by
some infinite computable set $R$.  Since   $A$ is
$\mathcal{D}$-maximal, by Definition \ref{def:Dmaxequiv}   there is an
infinite  c.e.\ set $D$, disjoint from $A$, such that either $A \sqcup D \supseteq^*
R$ or $D \cup A \cup R =^* \omega$. Without loss of generality, we may
assume the former, since in the latter case $A
\sqcup D \supseteq^* \overline{R}$.  Define $R_{n+1} = (D \cap R) - \bigsqcup_{i \leq n} R_i$. Since
$R$ splits $\overline{C}$ and $\overline{C} \cap R \subseteq R_{n+1} -
D_{n+1}$ it follows that $R_{n+1} - D_{n+1}$ is infinite.  By
definition $R_{n+1}$ is disjoint from $R_i$ for $i \leq n$ and as $D$
is disjoint from $A$ so is $R_{n+1}$.  Finally, as $D \cap R$ is the
complement of $(A \cap R) \cup \overline{R}$, $R_{n+1} = (D \cap R) -
\bigsqcup_{i \leq n} R_i$ is also computable.   

Since the $D_i$ are nested, we may suppose without loss of generality  that $A\sqcup D_0$ is $r$-maximal.  We now show that $A\sqcup D_0$ is atomless.     Suppose $W$ is a superset of $A \sqcup D_0$ such that $\overline{W}$ is
  infinite.  Since $A$ is $\mathcal{D}$-maximal and  $\{ D_0, D_1, \ldots \}$ is a Type 10 
  generating set consisting of nested sets,
  $W\subseteq^* A\sqcup D_j $ or $W\cup (A\sqcup D_j)=^*\omega$ for
  some $j$.  The latter case is impossible since
  $W\cup (A\sqcup D_j)=^*\omega$ implies there is a computable set $R$
  such that $W \cup R =^* \omega$, which contradicts that $A$ is
  $r$-maximal.  In the former case, $W\subseteq^* A\sqcup D_{j+1} $
  and $|A\sqcup D_{j+1} -W|=\infty$.\end{proof}

There are several examples in the literature of sets $A$ that are
$\mathcal{D}$-maximal splits of atomless $r$-maximal sets (see \S\ref{sec:type-10-sets}). 
 In \S\ref{sec:type-10-sets}, we will
construct a splitting of an atomless $r$-maximal set that has Type 10.

\subsubsection{ Types 7, 8 and 9: the hhsimple-like types}
\label{sec:types-7-8-9}

In this section, we discuss how some sets of Types 7, 8 and 9 behave
similarly to splits of hhsimple sets.  First, we show that we can
further refine generating sets for these Types.

\begin{lemma}\label{sec:types7}
  If a set $A$ is Type 7, there exists a Type 7 generating set
  $\{D_0, R_0, R_1, \ldots\}$ for $\mathcal{D}(A)$ such that:
  \begin{enumerate}
  \item for all $j\in\omega$, the set $R_j - D_0$ is infinite, and
    hence $\overline{A} -D_0$ is infinite.
  \item $ D_0 \subseteq \bigsqcup_{i \in \omega} R_i = \overline{A}$.
  \end{enumerate}
\end{lemma}
\begin{proof}
  Given some Type 7 generating set $\{D_0, R_0, R_1, \ldots\}$ for
  $\mathcal{D}(A)$, if $R_j \subseteq^* D_0$ for some $j$, we can
  remove $R_j$ from the list of generators. Infinitely many $R_j$ will
  remain since
  otherwise $A$ would be of lower Type. For the remaining $j$,
  $R_j - D_0$ is infinite. Then, by Lemma~\ref{sec:all}, we can adjust
  the $R_j$ so that
  $D_0 \subseteq \bigsqcup_{i \in \omega} R_i =
  \overline{A}$. \end{proof}

For sets of Type 8 or 9, for the first time, we will place conditions
on the order of the sets in the generating set.  We use this property
when we show that Type 8 and 9 $\mathcal{D}$-maximal sets exist.  The
proof, though more difficult than that of Lemma \ref{sec:types7} due
to this ordering, is similar to the proof of Lemma~\ref{sec:nested}.

\begin{lemma}\label{sec:types-7-8}
  If a set $A$ is Type 8 (respectively 9), there exists a Type 8
  (respectively 9) generating set such that for all $j \in \omega$:
  \begin{enumerate}
  \item \label{last} for all $i>j$, $D_i \cap R_j=\emptyset$ \\
    (respectively $(D_i-D_{i-1})\cap R_j=\emptyset$ for Type 9).
  \item\label{last2} the set $R_j - \bigcup_{i \leq j} D_i$ is
    infinite.
  \item $ \bigcup_{i \in \omega} D_i \subseteq \bigsqcup_{i \in
      \omega} R_i = \overline{A}$.\\
    So, $\overline{A} - \bigcup_{i \in \omega} D_i$ is infinite.
  \end{enumerate}

\end{lemma}

\begin{proof}
  Suppose that
  $\mathcal{G}=\{ \tilde{D}_0, \tilde{D}_1, \ldots \tilde{R}_0,
  \tilde{R}_1 \ldots \}$
  is Type 8 or Type 9 and generates $\mathcal{D}(A)$.  By
  Lemma~\ref{sec:all}, we can assume that
  $\overline{A}= \bigsqcup_{i \in \omega} \tilde{R}_i $. We
  inductively define a new generating set
  $\{D_0, D_1, \ldots, R_0, R_1, \ldots\}$ for $\mathcal{D}(A)$ with
  the desired properties and helper functions $d$ and $r$ from
  $\omega$ to $\omega$.

  Set $D_0= \tilde{D}_0$ and $d(0)= 0$. We claim that there exists an
  $i\in\omega$ such that $\tilde{R}_i- D_0$ is infinite.  This is true
  by definition if $\mathcal{G}$ is Type 9.  Suppose $\mathcal{G}$ is
  Type 8.  If the claim is false, then $\tilde{R}_i \subseteq^* D_0$
  for all $i\in\omega$.  So, $\{ \tilde{D}_0, \tilde{D}_1, \ldots\}$,
  a collection of pairwise disjoint c.e. sets, would generate
  $\mathcal{D}(A)$, and $A$ would be at most Type 6, a contradiction.
  Let $l$ be least such that $\tilde{R}_l - D_0$ is infinite. Let
  $R_0 = \bigsqcup_{j \leq l } \tilde{R}_j$ and $r(0)=l$.

  Assume that, for all $j\le i$, $D_j$, $R_j$, $d(j)$ and $r(j)$ are
  defined so that $D_i$ is not computable, $R_i$ is computable, and
  \begin{equation*}
    \bigcup_{j\le d(i)}\tilde{D}_j\ \cup\ \bigsqcup_{j\le r(i)}\tilde{R}_j\subseteq^* \bigcup_{j\le i} (D_j\cup R_j).
  \end{equation*}
  \noindent
  We claim there exists some (and hence a least) $l > d(i)$ such that
  $\bigcup_{d(i) < j \leq l} \tilde{D}_l - \bigsqcup_{j \leq i} R_j$
  is not computable.  If not, for all $k> d(i)$,
  \begin{equation*}\tilde{D}_k - \Big(
    \bigsqcup_{j \leq i} R_j \cup \bigcup_{d(i) < j < k} \tilde{D}_l
    \Big)
  \end{equation*} is computable. These computable
  sets, the computable sets $\{\tilde{R}_{r(i)+1},
  \tilde{R}_{r(i)+2}, \ldots \}$, and the noncomputable set $\bigcup_{j
    \leq i} ( D_j \cup R_j)$ generate $\mathcal{D}(A)$. Now, we can apply
  Lemma~\ref{sec:computable} to the computable sets in this list to show that $A$
  is at most Type 7. So, the desired least $l$ exists.
  Set $d(i+1) = l$ and $D_{i+1}=\bigcup_{d(i) < j \leq l} \tilde{D}_l -
  \bigsqcup_{j \leq i} \tilde{R}_j$.  If $\mathcal{G}$ is Type 9, we also add the elements of  $D_i$ to  $D_{i+1}$ to ensure the nesting property is satisfied.

  Let $l > r(i)$ be least such that
  $\tilde{R}_l - \bigcup_{j \leq i+1} D_j$ is infinite. Again, such an
  $l$ exists by definition if $\mathcal{G}$ is Type 9.  If
  $\mathcal{G}$ is Type 8 and $l$ fails to exist, none of the
  remaining $\tilde{R}_i$ are needed to generate $\mathcal{D}(A)$.
  Since the sets in $\{\tilde{D}_j\mid j\le i+1\}$ are pairwise
  disjoint ($\mathcal{G}$ is Type 8), $A$ is at most Type 6, a
  contradiction. So, we can set
  $R_{i+1} = \bigsqcup_{r(i) < j \leq l} \tilde{R}_j$ and $r(i+1)=l$.
  By construction, $\{D_0, D_1, \ldots, R_0, R_1, \ldots\}$ has the
  desired properties.
\end{proof}
Hence, if $A$ is Type 7, 8, or 9, we obtain the following analogue to
Theorem \ref{T:Lachlan}.

\begin{corollary}\label{sec:hhsimple-like-typesBA}
  Suppose that $A$ is Type 7, 8, or 9, and let
  $\breve{D}=\bigcup_{i\in\omega} D_i\cup A$.  Unless $A$ is of Type
  7, $\breve{D}$ is not a c.e.\ set.  The sets
  $\{ R_0, R_1, \ldots \}$ and finite boolean combinations of these
  sets form an infinite $\Sigma^0_3$ boolean algebra, $\mathcal{B}$,
  which is a substructure of
  $\mathcal{L}^*(\breve{D})$. \end{corollary}

\begin{proof}
  The relation $\subseteq^*$ is $\Sigma^0_3$. Each $R_i$ is
  complemented and infinitely different from $R_j$ for all
  $j\not=i$. \end{proof}

This substructure might be proper if $R_i \cap \breve{D}$ is finite
for all $i\in\omega$. If $\mathcal{B}$ is not proper then, for all
$i$, $\mathcal{L}^*(\breve{D} \cap \overline{R}_i)$ must be a boolean
algebra and hence $\breve{D}$ must be hhsimple inside $R_i$. By Lemma
\ref{sec:types-7-8}, $R_i \cap \breve{D}$ is a c.e.\ set.

If $A$ is $\mathcal{D}$-maximal, then the converse holds. Assume that
$A$ is $\mathcal{D}$-maximal and $\breve{D}$ is hhsimple inside $R_i$
for all $i$.  Given a set $W$, there is a finite set $F\subset\omega$
such that either
$W \subseteq^* \breve{D} \, \cup\, \bigsqcup_{i \in F} R_i$ or
$W \cup \breve{D}\, \cup\, \bigsqcup_{i \in F} R_i =^* \omega$. In
either case, $W$ is complemented inside
$\mathcal{L}^*(\breve{D})$. So, $\mathcal{L}^*(\breve{D})$ is a
boolean algebra.

In \S\ref{sec:build-hhsimple-like}, we show that $\mathcal{D}$-maximal
sets of Types 7, 8, and 9 exist. When we construct these three Types
of sets, we will ensure that $\breve{D}$ is hhsimple inside $R_i$ for
all $i\in\omega$, and, moreover, for all $ i \geq j$, $D_j \cap R_i$
is infinite and noncomputable. Our construction and Corollary
\ref{sec:hhsimple-like-typesBA} lead us to call Types 7, 8, and 9
hhsimple-like. For Type 7 we have the following corollary.

\begin{corollary}\label{sec:hk}
  There is half of a splitting of a hhsimple set that is
  $\mathcal{D}$-maximal and Type 7. \end{corollary}

Again, we can ask what kind of split is needed. By
Lemma~\ref{sec:splits2} and \ref{sec:splits1}, it cannot be a trivial
or Friedberg splitting. Note that Corollary~\ref{sec:hk} as presented
is known, see \citet[Theorem~4.1 (1)]{MR1264963}. In fact, Herrmann
and Kummer prove something stronger; see
\S\ref{sec:herrm-kumm-result}. They also directly prove that these
splits cannot be trivial or Friedberg.

\subsection{Questions}\label{Q0}  First, it is natural to ask as we did in Question \ref{Q:alltypesdefinable} if all the Types are definable.
In a related vein, it is natural to wonder whether Types 7, 8, 9, and
10 should be further subdivided.  We construct the
$\mathcal{D}$-maximal sets of Types 7, 8, and 9 very uniformly; for
all $i$, $\bigcup_{j \leq i} D_i$ is hhsimple inside $R_i$ and, for
all $i \geq j$, $D_j \cap R_i$ is infinite and noncomputable. Perhaps
one could further divide Types 7, 8, and 9 into finer types determined
by whether $\breve{D}$ is hhsimple inside $R_i$ or not, or, whether
for all $i \geq j$, $D_j \cap R_i$ is infinite and noncomputable, or
not. It is far from clear if this is productive. We suggest that the
reader look at \S\ref{sec:build-hhsimple-like} before considering
these questions.

We also asked in Question \ref{Q:converseType10} whether Type 10 sets
must be splits of  $r$-maximal sets.

\section{$\mathcal{D}$-maximal sets of all Types
  exist}\label{sec:main-result}
The next theorem is the main result of the paper.
\begin{theorem}\label{sec:main-result-1}
  There are complete and incomplete $\mathcal{D}$-maximal sets of each
  Type. Moreover, for any $\mathcal{D}$-maximal set $A$,
  \begin{enumerate}
  \item $A$ is maximal iff $A$ is Type 1.
  \item\label{sec:main-result-1(2)} There is a computable set $R$ such
    that $A \cup R$ is maximal (i.e. $A$ maximal inside
    $\overline{R}$) iff $A$ is Type 2.
  \item\label{sec:main-result-1(3)} $A$ is hemimaximal iff $A$ is Type
    3.
  \item $A$ is Herrmann iff $A$ is Type 4.
  \item $A$ is hemi-Herrmann iff $A$ is Type 5.
  \item $A$ has an $A$-special list iff $A$ is Type 6.
  \end{enumerate}
  So, for each of the first six Types, the $\mathcal{D}$-maximal sets
  of that Type form a single orbit. The $\mathcal{D}$-maximal sets of
  each of the remaining four Types break up into infinitely many
  orbits.
\end{theorem}

In \S\ref{firstsix}, we show that there are $\mathcal{D}$-maximal sets
of each of the first six types by proving the stronger statement in
the corresponding subcase of Theorem \ref{sec:main-result-1}.  The
orbits of the first five Types are known to contain complete and
incomplete sets, so we only need to address the Type 6 case to finish
the proof of Theorem \ref{sec:main-result-1} for the first six Types.

In \S\ref{sec:type-10-sets} we present a construction of
$\mathcal{D}$-maximal sets of Type 10 (by taking advantage of prior
work).  We also show that these sets break into infinitely many orbits
and that they can be of any noncomputable c.e.\ set.
In \S\ref{sec:build-hhsimple-like}, we construct
hhsimple-like $\mathcal{D}$-maximal sets of any noncomputable c.e.\
degree, i.e., Type 7, 8, and 9 $\mathcal{D}$-maximal sets.
We also prove that these sets break up into infinitely many orbits by
defining a further invariant on each of these Types.
It remains open, however, whether every $\mathcal{D}$-maximal set of
one of the last four Types is automorphic to a complete set.

\subsection{The first six parts of
  Theorem~\ref{sec:main-result-1}}\label{firstsix}


Recall that $A$ is simple iff $\mathcal{D}(A) = \{ \emptyset \}$. So,
a simple set $A$ is $\mathcal{D}$-maximal iff for all $W$ either
$W \subseteq^* A$ or $W \cup A =^* \omega$ iff $A$ is maximal. Hence,
a $\mathcal{D}$-maximal set is Type 1 iff it is maximal. We will need
the following lemma:
\begin{lemma}[\citet{SomeOrbits}] \label{sec:first-six-partsfried}
  Every nontrivial splitting of a $\mathcal{D}$-maximal set is a
  Friedberg splitting.
\end{lemma}
By Lemmas~\ref{sec:splits2}, \ref{sec:splits1} and
\ref{sec:first-six-partsfried}, a set $A$ is $\mathcal{D}$-maximal and
$\{X\}$ generates $\mathcal{D}(A)$ iff for all sets $W$ either
$W \subseteq^* A \sqcup X $ or $W \cup (A \sqcup X) =^* \omega$ iff
$A \sqcup X$ is maximal. Hence, the first three subcases of
Theorem~\ref{sec:main-result-1} hold.


\begin{lemma}\label{DmaxType4}\

  \begin{itemize}
  \item[(i)] A set $A$ is $\mathcal{D}$-maximal and Type 4 iff $A$ is
    Herrmann.
  \item[(ii)] A set $A$ is $\mathcal{D}$-maximal and Type 5 iff $A$ is
    hemi-Herrmann.
  \end{itemize}
\end{lemma}
\begin{proof} (i) \ ($\Rightarrow$) Suppose $A$ is a
  $\mathcal{D}$-maximal Type 4 set.  We show that $A$ is strongly
  $r$-separable.  Let $B$ be a set disjoint from $A$.  By assumption
  and Lemma \ref{sec:all}, there exist pairwise disjoint computable
  sets $R_1, \ldots, R_n, R_{n+1}$ belonging to a generating set for
  $\mathcal{D}(A)$ such that $B\subseteq\bigsqcup_{1\le i\le n} R_i$.
  The computable set $C = \bigsqcup_{1\le i\le n+1} R_i$ witnesses
  that $A$ is strongly $r$-separable.  ($\Leftarrow$) Given a Herrmann
  set $A$, we inductively construct a Type 4 generating set for
  $\mathcal{D}(A)$ as follows.  Suppose $ D_0, D_1, \ldots, D_n $ are
  pairwise disjoint computable sets that are all disjoint from $A$. If
  $W_n$ is disjoint from $A$, set
  $D = W_n \cup \bigsqcup_{i \leq n} D_i$, and otherwise, set
  $D =\bigsqcup_{i \leq n} D_i$. Since $A$ is strongly $r$-separable,
  there exists a computable set $C$ such that $D\subseteq C$ and
  $C- D$ is infinite. Setting
  $D_{n+1} = C - \bigsqcup_{i \leq n} D_i=C\cap\overline{\bigsqcup_{i
      \leq n} D_i}$
  completes the construction.  Also note that $A$ is not Type 1, 2, or
  3, since sets of those Type are not strongly $r$-separable.

\noindent
(ii) By Lemmas~\ref{sec:splits1}, \ref{sec:first-six-partsfried}, and
\ref{DmaxType4} (i), the hemi-Herrmann sets are $\mathcal{D}$-maximal
of Type 5. The other direction is straightforward.  Recall that 
hemi-Herrmann  and Herrmann sets each form their own orbit. Hence a
hemi-Herrmann set cannot have Type 4.
\end{proof}


\begin{lemma}
  A set $A$ is $\mathcal{D}$-maximal and  Type 6   iff $A$ has an $A$-special list. \end{lemma}

\begin{proof} ($\Leftarrow$) Note that if $\{ A, D_0, D_1, \ldots \}$
  is an $A$-special list and a set $W$ is disjoint from $A$, then
  $W \subseteq^* \bigsqcup_{ l\leq i}D_l$.  Otherwise, the condition
  $A\sqcup(W \cup \bigsqcup_{l\leq i}D_l) =^* \omega$ would hold,
  implying that $A$ would be computable.  ($\Rightarrow$) Given a
  c.e.\ set $W$, either $W\subseteq^* A\sqcup D$ or
  $W\cup (A\sqcup D)=^*\omega$ for some c.e.\ set $D$ disjoint from
  $A$ by $\mathcal{D}$-maximality.  The set $D$ is contained in
  finitely many sets from the Type 6 generating set.  So, a
  $\mathcal{D}$-maximal set $A$ has a Type 6 generating set
  $\{ D_0, D_1, \ldots \}$ iff $\{ A, D_0, D_1, \ldots \}$ is an
  $A$-special list.  Recall that maximal, hemimaximal, Herrmann, and hemi-Herrman sets, as well as sets
with $A$-special lists form distinct definable orbits (see
\S\ref{S:introexDmax}), and that Types of generating sets are invariant.  These facts imply the result. \end{proof}

  Although it was previously shown that there
are complete Herrmann and hemi-Herrmann sets, it is not explicitly
shown in \citet{MR2366962} that a complete or incomplete set with an
$A$-special list exists.  In Remark~\ref{degree}, we discuss how the
construction found in \cite{MR2366962} of sets with $A$-special lists
can be modified to ensure the resulting set is complete or incomplete.

\subsection{$\mathcal{D}$-maximal sets of Type 10 and atomless
  $r$-maximal sets}\label{sec:type-10-sets}

Lerman and Soare constructed an atomless $r$-maximal set $A$ and a
nontrivial splitting $A_0 \sqcup A_1 = A$ so that $A_1 \cup (W - A)$
is c.e.\ for every coinfinite $W\in \mathcal{L}^*(A)$
in\cite[Theorem~2.15]{Lerman.Soare:80}.  Herrmann and Kummer proved
that such a split $A_0$ is $\mathcal{D}$-maximal
\cite[Proposition~4.5]{MR1264963}.  We need the following lemma:

\begin{lemma}\label{sec:atomless2}
  If a noncomputable set $A$ is half of a splitting of an atomless
  set, then $A$ is not half of a splitting of a maximal
  set. \end{lemma}

\begin{proof}
  Assume that $A \sqcup A_1$ is an atomless set and $A \sqcup A_2$ is
  maximal. Since $A \sqcup A_1$ is atomless, $A\sqcup A_1$ cannot be
  (almost) a subset of $A\sqcup A_2$. Since $A \sqcup A_2$ is
  maximal, $(A \sqcup A_2) \cup (A \sqcup A_1) =^* \omega$.
  Therefore, $A \sqcup (A_1 \cup A_2) =^* \omega$, and $A$ is
  computable. \end{proof}

By Lemma~\ref{sec:atomless2} 
and the first 3 subcases of Theorem \ref{sec:main-result-1}, $A_0$
does not have Type 1, 2 or 3.  Therefore, by
Lemma~\ref{sec:notsplitrmaxi}, $A_0$ is in fact a Type 10
$\mathcal{D}$-maximal set.

The construction of Lerman and Soare is a version of John Norstad's
construction (unpublished) that has been modified several times (see
\cite[Section X.5]{Soare:87}).  Here we briefly discuss how to alter
the construction in \citet[Section~2]{mr2001a:03087} to directly show
that 
$A_0$ is $\mathcal{D}$-maximal.
For the remainder of this section, we assume that the reader is
familiar with \cite{mr2001a:03087}.

As we enumerate $A$, we build the splitting $A =A_0 \sqcup A_1$.  All
the balls that are \emph{dumped} by the construction are added to
$A_1$.  Since $A_0$ would be empty without any other action, we add
requirements $\mathcal{S}_e$ to ensure that $A_0$ is not computable.
Specifically, we have
\begin{equation*}
  \tag*{$\mathcal{S}_{e}$:}
  W_e \neq \overline{A_0}.
\end{equation*}
We say that $\mathcal{S}_e$ is met at stage $s$ if there is an
$x \leq s$ such that $\varphi_{e,s}(x) = 1$ but $x \in A_{0,s}$.  We
also add a Part
III  to the construction in \cite[Construction~2.5]{mr2001a:03087}.\\

\noindent
\emph{Part III}: Let $x = d^s_{\langle e, 0 \rangle}$.  If
$\mathcal{S}_e$ is met or $x$ has already been dumped into $A$ at
stage $s$, do nothing.  Otherwise, if $\varphi_{e,s}(x) = 1$, add $x$
to
$A_0$ and realign the markers as  done in Parts I and II.\\

It straightforward to show that $\mathcal{S}_e$ is met and that Part
III does not impact the rest of the construction.  So, it is left to
show that $A_0 $ is $\mathcal{D}$-maximal. 
By requirement $P_e$
and \cite[Lemma~2.3]{mr2001a:03087}, either $W_e
\subseteq^*H_e$ or $\overline{A} \subseteq^*
W_e$.  In the latter case, $A_0 \cup A_1 \cup W_e =^*
\omega$.  So assume that $W_e \subseteq^*
H_e$.  It is enough to show that $W_e
-A_0$ is a c.e.\ set.  \cite[Definition 2.9, Lemma
2.11]{mr2001a:03087} provides a c.e.\ definition of $H_e$.
To guarantee that $W_e
-A_0$ is c.e., we have to slightly alter the definition of
$s'$
in \cite[Definition 2.9]{mr2001a:03087}.  In particular, choose
$s'$
so that if $\mathcal{S}_i$
will be met at some stage, then it is met by stage $s'$,
for all $i
\leq e$.  This change at most increases the value of
$s'$.
This alteration in $s'$,
\cite[Lemma 2.10]{mr2001a:03087}, and the construction together imply
that $H_e \searrow A_0$ is empty.  Since $W_e \subseteq^* H_e$,
 $$((H_e \setminus A_0)\cap W_e) \cup
 (W_e \cap A_1) =^* W_e -
 A_0.$$ Hence, $W_e - A_0$ is c.e.\ as required.

 This construction of $A_0$
 clearly mixes with finite restraint; rather than using $d^s_{\langle
   e, 0 \rangle}$ for $\mathcal{S}_e$ use the least $d^s_{\langle e, j
   \rangle}$ above the restraint.
 To code any noncomputable c.e.\ set $X$
 into $A_0$,
 we have to alter the dumping slightly.  If a ball $x=
 d^s_{\langle e, 0 \rangle}$ is dumped into
 $A$,
 always add it to $A_0$.
 All other dumped balls go into $A_1$.
 Now if $e$
 enters $X$
 at stage $s$,
 also add $d^s_{\langle
   e, 0 \rangle}$ into $A_0$.  It is not hard to show that
 $A_0$
 computes $X$,
 just alter the above $s'$
 in the c.e.\ definition of $H_e$
 so that $X
 \restriction e+1 = X_{s'} \restriction
 e+1$).  Since these versions of coding and finite restraint mix, we
 can construct $A$ of any noncomputable c.e.\ degree.

 Cholak and Nies \cite[Section 3]{mr2001a:03087} go on to construct
 infinitely many atomless $r$-maximal sets $A^n$ that all reside in
 different orbits.  We use the ideas there together with our modified
 construction to obtain $A^n = A^n_0 \sqcup A^n_1$.  We claim that the
 sets $A^n_0$ also fall into infinitely many distinct orbits. Assume
 that $A^n_0 \sqcup B$ is an atomless $r$-maximal set. Since $A^n_0$
 is not computable, $A^n_0 \sqcup (A^n_1 \cup B ) \neq^* \omega$. A
 $T^{n+1}$-embedding of $\mathcal{L}^*(A^{n+1})$ into
 $\mathcal{L}^*(A^n_0 \sqcup (A^n_1 \cup B ))$ would provide a
 $T^{n+1}$-embedding of $\mathcal{L}^*(A^{n+1})$ into
 $\mathcal{L}^*(A^n_0 \sqcup A^n_1)$.  By \cite[Lemma~3.5,
 Theorem~3.6]{mr2001a:03087}, the latter cannot exisit so neither can
 the former.
 In $\mathcal{L}^*(A^n_0 \sqcup A^n_1)$, $B$ is contained by some
 $H^n_e$, where $e= i_{0^m}$, for some $m$ (see
 \cite[Theorem~2.12]{mr2001a:03087}).
 By definition of $T^n$, the tree above $0^m$ is isomorphic to $T^n$.
 So there is a $T^{n}$-embedding of $\mathcal{L}^*(A^n)$ into
 $\mathcal{L}^*(A^n_0 \sqcup (A^n_1 \cup B))$ and hence into
 $\mathcal{L}^*(A^n_0 \sqcup B)$.  Thus, none of the $A^n_0$ belong to
 the same orbit.

\section{Building hhsimple-like $\mathcal{D}$-maximal
  sets}\label{sec:build-hhsimple-like}
We continue with the proof of Theorem \ref{sec:main-result-1}.  We
construct $\mathcal{D}$-maximal sets of Types 7, 8, and 9 and show
that the collection of sets of each of these Types breaks up into
infinitely many orbits.

In \S\ref{sec:types-7-8-9}, we discussed how sets of Types 7, 8, and 9
are like hhsimple sets.  Lachlan's construction in the second half of
Theorem \ref{T:Lachlan} serves as the backbone of our constructions,
but we also use it modularly within these constructions.
Our approach is to treat this theorem as a blackbox.

In \S\ref{SS:Types7-9Overview}, we describe how to construct a set $H$
that is close to being hhsimple and is associated with a boolean
algebra with a particularly nice decomposition.  In \S
\ref{SS:requirements}, we add requirements ensuring that the
construction in \S\ref{SS:Types7-9Overview} results in a hhsimple-like  set
with a $\mathcal{D}$-maximal split of Type 7, 8, or 9.
\subsection{Herrmann and Kummer's Result}\label{sec:herrm-kumm-result}

It is important to note that Herrmann and Kummer \cite[Theorem~4.1
(1)]{MR1264963} already constructed $\mathcal{D}$-maximal splits of
hhsimple sets.  In fact, their result is stronger than the result
presented here, in the sense that, given any infinite $\Sigma_3^0$
boolean algebra $\mathcal{B}$, they provide a construction of a
$\mathcal{D}$-maximal split of a hhsimple set of flavor $\mathcal{B}$.
Although Herrmann and Kummer show that their split of a hhsimple is,
in our language, not of Type 1, 2, or 3, they do not further
differentiate between sets of Type 7, 8, or 9.  Furthermore, they do
not show that the collections of such sets break into infinitely many
orbits, as we do here.

The proof of \cite[Theorem~4.1 (1)]{MR1264963} is rather difficult and
spans several papers, including \cite{MR884722} and
\cite{MR1034561}. These papers together provide a fine analysis of
Lachlan's result and of decompostions of infinite boolean algebras.
This analysis is in terms of $\Sigma^0_3$ ideals of $2^{< \omega}$,
and the proof of \cite[Theorem~4.1 (1)]{MR1264963} divides into three
cases based on
structural properties of the given $\Sigma^0_3$ ideal.

We claim it is possible to obtain Herrmann and Kummer's result via a
modification of the construction below by translating their work into
the language of boolean
algebras.  
However, since this general approach would increase the complexity of
the proof and our goals are different, we focus on sets corresponding
to boolean algebras with especially nice decompositions.

%

\subsection{Background on Small Major Subsets}\label{sec:smallness1}

We need some background on smallness and majorness for our
construction.  These notions will be used in \S\ref{SS:Type9SM} and
\S\ref{sec:smallness}. One can delay reading this section until then.

Smallness and majorness were introduced by Lachlan in
\cite{Lachlan:68*2} and further developed in \cite{Stob:79}. See also
\cite[X.4.11]{Soare:87}, \cite{Maass.Stob:83}, and
\cite{mr2004f:03077} for more on these concepts.

\begin{definition} Let $B$ be a c.e.\ subset of a c.e.\ set $A$.  We
  say that $B$ is a \emph{small} subset of $A$ if, for every pair of
  c.e.\ sets $X$ and $Y$, $X \cap (A-B) \subseteq^* Y$ implies that
  $Y \cup (X-A)$ is a c.e.\ set.
\end{definition}

\begin{definition} Let $C$ be a c.e.\ subset of a c.e.\ set $B$.  We
  say that $C$ is \emph{major} in $B$, denoted $C \subseteq_m B$, if
  $B - C$ is infinite and for every c.e.\ set $W$, the containment
  $\overline{B} \subseteq^* W$ implies $\overline{C} \subseteq^* W$.
\end{definition}

We need the following straightforward results about small major
subsets.  Note that any c.e.\ subset of computable set is small in the
computable set.

\begin{lemma}\label{stob1}\
  Let $E$ and $F$ be subsets of $D$, and let $R$ be a computable set.
  \begin{enumerate}
  \item\label{S11} (Stob \cite{Stob:79}) Suppose $E$ is small in $D$.
    If $D \subseteq \Dhat$, then $E$ is small in $\Dhat$.  Similarly,
    if $\Ehat \subseteq E$, then $\Ehat$ is small in $D$.
  \item\label{S12} (Stob \cite{Stob:79}) If $E$ is small in $D$, then
    $E \cap R$ is small in $D\cap R$.

  \item\label{S26} If $E$ is major in $D$, then $E\cap R=^*D\cap R$ or
    $E\cap R$ is major in $D\cap R$.

  \item\label{S24} If $F$ is major in $E$ and $E$ is major in $D$,
    then $F$ is major in $D$.

  \item\label{S23} If $E$ is major in $D$ then $E$ is simple inside
    $D$.

  \item\label{S25} If $E$ is major in $D$ and $D$ is hhsimple, then
    every hhsimple superset of $E$ contains $D$.

  \end{enumerate}
\end{lemma}

\begin{proof}
  (\ref{S11}), (\ref{S12}) The proofs of these statements can be found
  in \cite{mr2004f:03077}.
  \noindent
  (\ref{S26}) 
  If $\overline{D\cap R}=\overline{D}\cup\overline{R}\subseteq^* W$,
  then $\overline{E}\cup\overline{R}=\overline{E\cap R}\subseteq^* W$.

  \noindent
  (\ref{S24}) If $\overline{D} \subseteq^* W$, then
  $\overline{E} \subseteq^* W$ and, hence,
  $\overline{F} \subseteq^* W$.

  \noindent
  (\ref{S23}) Suppose that there is an infinite c.e.\ set
  $W \subseteq^* (D- E)$. Then, there is an infinite computable set
  $R \subseteq^* (D-E)$ such that
  $\overline{D} \subseteq^* \overline{R}$ but
  $\overline{E} \nsubseteq^* \overline{R}$.

 \noindent
 (\ref{S25}) Let $H$ be a hhsimple superset of $E$. Then, there is a
 c.e.\ set $W$ such that $H \subseteq^* W$, $D \cup W = \omega$, and
 $W \cap D \subseteq^* H$. So, $\overline{D} \subseteq^* W$. If $D-H$
 is infinite, $\overline{E} \nsubseteq^* W$, a contradiction. So,
 $D \subseteq^* H$.

\end{proof}

The following theorem by Lachlan will be very useful:

\begin{theorem}[\citet{Lachlan:68*2} (also see {\cite[X
    4.12]{Soare:87}})]
  \label{sec:smallmajorsubset}
  There is an effective procedure that, given a noncomputable c.e.\
  set $W$, outputs a small major subset of $W$.
\end{theorem}

\subsection{Construction overview}\label{SS:Types7-9Overview}
We now construct $\mathcal{D}$-maximal sets $A$ of Type 7, Type 8, and
Type 9.  These constructions are very similar to the construction of
splits of a hhsimple set $H$ of a given kind of flavor $\mathcal{B}$.
Let $\mathcal{B}$ be a $\Sigma_3^0$ boolean algebra with infinitely
many pairwise incomparable elements.  We call a subset
$\{b_i\}_{i\in\omega}$ of $\mathcal{B}$ a \emph{skeleton} for
$\mathcal{B}$ if the elements in $\{b_i\}_{i\in\omega}$ are pairwise
incomparable and, for every element of $\mathcal{B}$, either it or its
complement is below the join of finitely many elements in
$\{b_i\}_{i\in\omega}$.  If $\{b_i\}_{i\in\omega}$ is a skeleton for
$\mathcal{B}$ and $\mathcal{B}_{b}:= \mathcal{B}\restriction [0, b]$
for any $b\in\mathcal{B}$, then
$\mathcal{B} = \join_{i\in\omega} \mathcal{B}_{b_i}$.

For the remainder of \S\ref{sec:build-hhsimple-like}, we fix an
arbitrary $\Sigma_3^0$ boolean algebra $\mathcal{B}$ that has a
computable skeleton $\{b_i\}_{i\in\omega}$.  We show how to construct
a set $H$ that is hhsimple (or, at first, close to hhsimple) with
flavor $\mathcal{B}$.  (Our construction can be made to work for any
boolean algebra with a $\mathbf{0''}$-computable skeleton, but the
added complexity does not gain us a sufficiently better result.)

In the Type 7 case, we intend to build a splitting $H=A\sqcup D$ so
that $\mathcal{D}(A)$ is generated by $D=\bigsqcup_{i\in\omega} H_i$
and a list $\{R_i\}_{i\in\omega}$ of pairwise disjoint infinite
computable sets where $H_i\subset R_i$ has flavor $\mathcal{B}_{b_i}$
inside $R_i$ for all $i\in\omega$.
Specifically, we build these objects via Lachlan's construction
(Theorem \ref{T:Lachlan}) so that $\mathcal{E}^*(R_i-H_i)$ is
isomorphic to $\mathcal{B}_{b_i}$. (Note that $\mathcal{E}^*(R_i-H_i)$
should be thought of as the collection of c.e.\ supersets of $H_i$
that are contained in $R_i$.)
Then, $\mathcal{B}$ is isomorphic to a (possibly proper) substructure
of $\mathcal{L}^*(H)$. The structures $\mathcal{B}$ and
$\mathcal{L}^*(H)$ are isomorphic if, in addition, for every c.e.\ set
$W$ there exists an $n\in\omega$ so that
$W\subseteq^* H\cup \bigsqcup_{i\le n} R_i$ or
$W\cup H\cup \bigsqcup_{i\le n} R_i=^*\omega$.

We make two remarks.  First, although Lachlan's construction can be
done uniformly inside any computable set, the list
$\{R_i\}_{i\in\omega}$ we construct will not be uniformly computable.
Hence, we must ensure that $H$ is a c.e.\ set.  Second, since
$\mathcal{L}^*(H)$ is a boolean algebra, for every c.e.\ superset $W$
of $H$ there is a c.e.\ set $\tilde{W}$ such that
$W \cup \tilde{W} \cup H = \omega$ and $W \cap \tilde{W} \subseteq
H$.
So, there is a computable set $R$ such that
$R \cap \overline{H} = W \cap \overline{H}$. Thus,
if we construct $H$ and $\{R_i\}_{i\in\omega}$ with the properties
detailed above,
$R$ or $\overline{R}$ is contained in the union of a finite subset of
$\{R_i\}_{i\in\omega}$ for any computable superset $R$ of $H$.  Note
that construction of lists like $\{R_i\}_{i\in\omega}$ appeared in
some form in many constructions by Cholak and his coauthors and
others, e.g., \citet{MR0422004}.

\subsection{Requirements}\label{SS:requirements}
\subsubsection{$\mathcal{D}$-maximal sets of Type 7}

We formally state the requirements necessary to construct a
$\mathcal{D}$-maximal set $A$ such that $A\sqcup D$ is a splitting of
a hhsimple set $H$ of flavor $\mathcal{B}$.  As mentioned above, we
simultaneously construct a pairwise disjoint list of infinite
computable sets $R_i$ that are all disjoint from $A$ and sets $H_i$
contained in $R_i$ so that the union of $A$ and
$D =\bigsqcup_{i\in\omega} H_i$ equals $H$.  We require that these
objects satisfy the requirements:

\begin{equation*}
  \tag*{$\mathcal{R}_{e}$:}
  W_e \subseteq^* A \cup D \cup \bigsqcup_{i \le e} R_i
  \text{ or } {W_e} \cup A \cup D \cup \bigsqcup_{i\leq e} R_i  =^* \omega,
\end{equation*}

\begin{equation*}
  \tag*{$\mathcal{S}_{e}$:}
  \overline{A} \neq W_e,
\end{equation*}
and
\begin{equation*}
  \tag*{$\mathcal{L}_{i}$:}
  \mathcal{E}^*(R_i-H_i) \text{ is isomorphic to }\mathcal{B}_{b_i}.
\end{equation*}

We satisfy the $\mathcal{S}_e$ requirements as usual, and they imply
that $A$ is not computable.  We satisfy the $\mathcal{L}_i$
requirements by applying Lachlan's construction.  The $\mathcal{R}_e$
requirements ensure that $A$ is $\mathcal{D}$-maximal and that
$\{D\}\cup\{R_i\}_{i\in\omega}$ generates $\mathcal{D}(A)$ (if $D$ is
a c.e.\ set).  Taken together, the $\mathcal{R}_e$ and $\mathcal{L}_i$
requirements guarantee that $\mathcal{L}^*(A \cup D)$ is isomorphic to
$\mathcal{B}$.  The $\mathcal{R}_e$ requirements take some work, as
does ensuring that all constructed sets are computably enumerable.

%

%

%



\subsubsection{$\mathcal{D}$-maximal sets of Types 8 and
  9}\label{Dmaxreq89}

To construct a $\mathcal{D}$-maximal set of either Type 8 or 9, we
must construct a generating set for $\mathcal{D}(A)$ of the proper
form $\{D_0, D_1, \ldots, R_0, R_1, \ldots\}$.  This generating set
contains infinitely many properly c.e.\ sets rather than a single
properly c.e.\ set as in the Type 7 case. Hence, we must modify the
$\mathcal{D}$-maximality {$\mathcal{R}_e$} requirements for these
cases.

\begin{equation*}
  \tag*{$\mathcal{R}'_{e}$:}
  W_e \subseteq^* A \cup \bigcup_{i \le e} D_i  \cup \bigsqcup_{i \le e} R_i
  \text{ or } {W_e} \cup A \cup \bigcup_{i \leq e} D_i \cup
  \bigsqcup_{i\leq e} R_i  =^* \omega.
\end{equation*}

We still construct the lists $\{R_i\}_{i\in\omega}$ and
$\{H_i\}_{i\in\omega}$ as in the Type 7 case.  In the Type 8 case, we
now use the Friedberg Splitting Theorem to break $H_i$ into $i+1$
infinite disjoint sets $H_{i,j}$ for $0\le j \leq i$. Then, we let
$D_j = \bigsqcup_{i\in \omega, i\ge j} H_{i,j}$, and we ensure that
$D_j$ is c.e.\ by construction.  Note that $D_j\cap R_i=\emptyset$ if
$i<j$ and the list $\{D_i\}_{i\in\omega}$ is pairwise disjoint.

In the Type 9 case, we use the $H_i$ to construct the nested list of
c.e.\ sets $\{D_i\}_{i\in\omega}$ so that for all $i\in\omega$:
\begin{enumerate}
\item\label{T9C1} $D_i \cap R_j= D_j\cap R_j=H_j$ for $j\le i$,
\item\label{T9C3} $D_i\cap \overline{\bigsqcup_{j\le i} R_j}$ is
  simple inside $D_{i+1}\cap \overline{\bigsqcup_{j\le i} R_j}$, so
  $(D_{i+1}- D_i)\cap \overline{\bigsqcup_{j\le i} R_j}$ contains no
  infinite c.e.\ sets.
\end{enumerate}

\begin{remark}\label{R:SimpleSM} Observe that conditions (\ref{T9C1})
  and (\ref{T9C3}) imply that for any $l$, either $D_{i}=^*D_{i+1}$ on
  $R_l$ or $D_i$ is simple inside $D_{i+1}$ on $R_l$.  Hence,
  $(D_{i+1}-D_i)\cap R_l$ contains no infinite c.e.\ sets.

\end{remark}

Let $\breve{D} =\bigcup_{i \in\omega } D_i $.  In both the Type 8 and
9 cases,
$$\breve{D} \cap \bigsqcup_{i \le e } R_i=\bigcup_{i \le e } D_i \cap
\bigsqcup_{i \le e } R_i $$ by the descriptions above.

\subsubsection{Type 9 and small majorness}\label{SS:Type9SM} To ensure that property
(\ref{T9C3}) holds
in the Type 9 case, we satisfy the following requirements.  (See
\S\ref{sec:smallness1} for definitions.)
\begin{equation*}
  \tag*{$\mathcal{I}_{i}$:} D_i\cap \overline{\bigsqcup_{j\le i} R_j}  \text{ is a \emph{small major}
    subset of }D_{i+1}\cap \overline{\bigsqcup_{j\le i} R_j}
\end{equation*}

We use Lachlan's Theorem~\ref{sec:smallmajorsubset} to modularly to
meet $\mathcal{I}_{i}$ (see Lemma~\ref{D_esmD_{e+1}} for the proof).


\subsection{Sufficiency of requirements}
If the requirements listed in \S\ref{SS:requirements} are met as
described, the set $A$ certainly will be a $\mathcal{D}$-maximal set
of Type at most 7, 8, or 9 respectively (since $\mathcal{D}(A)$ has a
generating set of that Type).  However, we also must ensure that
$\mathcal{D}(A)$ does not have lower Type.

In the following, we examine the Type 7, 8, and 9 cases together as
much as possible. To do so and for notational simplicity, in the Type
7 case, set $D_0=D$ and $D_i=\emptyset$ for all $i\not=0$.

\subsubsection{Not Type 1, 2, 3 or 10} First, note that the
requirements $\mathcal{S}_i$ guarantee that $A$ is not
simple. Therefore, $\mathcal{D}(A)$ is not Type 1 by
Lemma~\ref{sec:gener-sets-simple}.  If $A$ is Type 2 or 3, there is a
c.e.\ set $W_e$ such that $A \sqcup W_e$ is maximal by
Theorem~\ref{sec:main-result-1} (\ref{sec:main-result-1(2)}) and
(\ref{sec:main-result-1(3)}).  Assume that $W_e$ is disjoint from
$A$. By requirement $\mathcal{R}'_e$, either
$W_e \subseteq^* \bigcup_{i\le e } D_i \cup \bigsqcup_{i \le e} R_i$
or
$A\sqcup ({W_e} \cup \bigcup_{i \leq e} D_i \cup \bigsqcup_{i\leq e}
R_i) =^* \omega$.
The latter case implies that $A$ is computable.  Since $A$ is not
computable by the requirements $\mathcal{S}_i$, the latter case cannot
hold. In the former case, the set $R_e\sqcup A \sqcup W_e$ witnesses
that $A \sqcup W_e$ is not maximal (or even $r$-maximal). Therefore
$A$ is not Type 2 or 3.  By definition, the set $A$ is not Type 10
(since $\mathcal{D}(A)$ has a generating set of Type 7, 8, or
9). 
Thus, $A$ is not Type 1, 2, 3, or 10.

\subsubsection{A Technical Lemma}

We need the following lemma to show that the sets we construct are not
of lesser Type.  Lemma \ref{sec:one-needed-lemma} is the one place
where we use that these Types are constructed very uniformly, as
mentioned in \S\ref{Q0}. It is unclear how to separate these Types
otherwise.

\begin{lemma}\label{sec:one-needed-lemma}
  Let $\breve{D} = \bigcup_{i\in\omega} D_i$. Let $W_e$ be disjoint
  from $A$. Then, $W_e \subseteq^* \breve{D}$ or $W_e - \breve{D}$ is
  not a c.e.\ set. Moreover, for the Type 8 and 9 cases,
  $W_e \subseteq^* \bigcup_ {i\le e} {D}_i$ or $W_e - D_i$ is not
  c.e.\ for all $i\le e$.

\end{lemma}

\begin{proof}
  By requirement $\mathcal{R}'_e$, either
  $W_e \subseteq^* \bigcup_{i\le e } D_i \cup \bigsqcup_{i \le e} R_i$
  or \linebreak
  $A \sqcup (W_e\cup \bigcup_{i\le e } D_i \cup \bigsqcup_{i \le e}
  R_i)=^*\omega$.
  Since requirements $\mathcal{S}_i$ ensure that $A$ is noncomputable,
  the latter statement cannot hold.  So,
  $W_e \subseteq^* \bigcup_{i\le e } D_i \cup \bigsqcup_{i \le e} R_i$
  and thus $W_e-\breve{D}\subseteq^* \bigsqcup_{i \le e} R_i$. Suppose
  that $W_e\not \subseteq^* \bigcup_ {i\le e} {D}_i\subset \breve{D}$.

  By requirement $\mathcal{L}_i$, the set $H_i$ is hhsimple inside
  $R_i$.  Therefore the set
  $$\breve{D} \cap \bigsqcup_{i \leq e } R_i=\bigcup_{i \leq e } D_i
  \cap \bigsqcup_{i \leq e } R_i$$
  is hhsimple inside $\bigsqcup_{i \leq e } R_i$. So,
  $W_e - \breve{D}=W_e- \bigcup_{i\le e}D_i$ is not a c.e.\ set.

  For the Type 8 case, recall that $D_0, D_1, \ldots, D_i$ form a
  Friedberg splitting of their union inside $R_i$.  Hence, $W_e-D_i$
  is not a c.e.\ set for all $i\le e$.

  For the Type 9 case, we argue by reverse induction.  Since
  $D_{e} = \bigcup_{i\le e}D_i$ (these sets are nested), $W_e -D_{e} $
  is not a c.e.\ set.   Since $W_e-D_e\subseteq^*\bigsqcup_{i \leq e } R_i$, there is some $i'\le e$ such that $(W_e-D_e)\cap R_{i'}$ is not a c.e.\ set.   Assume that $ (W_e - D_{j+1})\cap R_{i'}$
  is not c.e.\ for $j +1 \le e$ 
  (and, so, is
  infinite).
    Suppose $(W_e-D_j)\cap R_{i'}$ is a c.e.\ set.  
   Since
  $(W_e - D_{j+1})\cap R_{i'}$ is not c.e., $D_j$ is not almost equal to
  $D_{j+1}$ on $R_{i'}$.  The c.e.\ set
  $(W_e - D_{j})\cap D_{j+1}\cap R_{i'}$ is infinite and witnesses that
  $D_j$ is not simple inside $D_{j+1}$ on $R_{i'}$, contradicting
  Remark~\ref{R:SimpleSM}.  So, $(W_e-D_j)\cap R_{i'}$ and  $W_e-D_j$ are not c.e.\ sets.   Therefore, $W_e - D_i$ is not c.e.\ for all 
  $i\le e$. \end{proof}

\subsubsection{Not Type 4, 5, or 6} Now assume that the
$\mathcal{D}$-maximal set $A$ constructed has a generating set for
$\mathcal{D}(A)$ of Type 4, 5, or 6.  Since $A$ is
$\mathcal{D}$-maximal, $D_0$ is almost contained in the union of
finitely many of these generators. Then, there is another infinite
generator $W_e$ in this generating set almost disjoint from $D_0$.
The fact that $W_e-D_0$ is c.e.\ contradicts Lemma
\ref{sec:one-needed-lemma}.

\subsubsection{Type 8 is not Type 7}

Suppose that $\{ \tilde{D}, \tilde{R}_0, \tilde{R}_1, \ldots \}$ is a
Type 7 generating set for $\mathcal{D}(A)$, where $A$ is constructed
via the Type 8 construction described above.  By construction, there
is an $e$ such that
$\tilde{D} \subseteq^* \bigsqcup_{i \le e} R_i \cup \bigsqcup_{i \le
  e} D_i$,
so $D_{e+1}$ and $\tilde{D}$ are almost disjoint. Then, there is an
$l$ such that
$\bigsqcup_{i \le e} R_i \cup \bigsqcup_{i \leq e+1} D_i \subseteq^*
\tilde{D} \cup \bigsqcup_{i\le l} \tilde{R}_i$;
so, $D_{e+1}\subseteq^*\bigsqcup_{i\le l} \tilde{R}_i$.  Finally,
there is a $k$ such that
$\tilde{D} \cup \bigsqcup_{i\le l} \tilde{R}_i \subseteq^*
\bigsqcup_{i \le k} R_i \cup \bigsqcup_{i \le k} D_i$.
Now, by construction, $R_{k+1} - \breve{D}=R_{k+1}-H_{k+1}$ is
infinite. Hence, there is an $m > l$ such that
$\tilde{R}_m - \breve{D}$ is infinite. Observe that $\tilde{R}_m$ is
disjoint from $D_{e+1}$ and $A$. Thus,
$\tilde{R}_m - D_{e+1} = \tilde{R}_m$ is an infinite c.e.\ set,
contradicting Lemma \ref{sec:one-needed-lemma}.

\subsection{Small Major Subsets and Type 9 Sets}\label{sec:smallness}

In order to show that the set $A$ resulting from the construction
outlined for the Type 9 case is not of Type 7 or Type 8, we need the
following lemma.

\begin{lemma}\label{D_esmD_{e+1}}
  Suppose we obtain the lists $\{D_i\}_{i\in\omega}$ and
  $\{R_i\}_{i\in\omega}$ while constructing a $\mathcal{D}$-maximal
  set $A$ according to the Type 9 requirements outlined in \S
  \ref{SS:requirements}. The following statements hold for $j\le i$.

  \begin{enumerate}
  \item\label{D_esmD_{e+1}2} Either $D_j=^*D_{i}$ on
    $\overline{\bigsqcup_{l\le i}R_l}$ or
    $D_j\cap \overline{\bigsqcup_{l\le i}R_l}$ is small major in
    $D_i\cap \overline{\bigsqcup_{l\le i}R_l}$. In the latter case,
    $D_j$ is simple inside $D_i$ on
    $\overline{\bigsqcup_{l\le i}R_l}$.

  \item\label{D_esmD_{e+1}1} For $l\le i$, either $D_j=^*D_{i}$ on
    $R_l$ or $D_j$ is a small major subset of $D_i$ on $R_l$.  In the
    latter case, $D_j$ is simple inside $D_i$ on
    $R_l$.
  \end{enumerate}
\end{lemma}

\begin{proof}
  We prove (\ref{D_esmD_{e+1}2}) by induction on $i\ge j$.  The base
  case $i=j$ holds trivially.  Suppose the statement holds for
  $i\ge j$.
  Requirement $\mathcal{I}_i$ and Lemma~\ref{stob1} (\ref{S11}),
  (\ref{S24}) imply that $D_j\cap \overline{\bigsqcup_{l\le i}R_l}$ is
  small major in $D_{i+1}\cap \overline{\bigsqcup_{l\le i}R_l}$.  The
  result follows by Lemma~\ref{stob1} (\ref{S12}), (\ref{S26}).

  The proof of (\ref{D_esmD_{e+1}1}) is similar but also uses the
  construction property that $D_i \cap R_j= D_j\cap R_j$ for $j<i$ and
  Lemma~\ref{stob1} (\ref{S26}).  The second half of both statements
  holds by Lemma~\ref{stob1} (\ref{S23}).
\end{proof}

\subsubsection{Type 9 not Type 7}
We now show that the $\mathcal{D}$-maximal set $A$ obtained via the
Type 9 construction is not Type 7.  Assume that
$\{ \tilde{D}, \tilde{R}_0, \tilde{R}_1, \ldots \}$ is a Type 7
generating set for $\mathcal{D}(A)$. By the $\mathcal{R}_e'$
requirements, there is some $e$ such that
$\tilde{D} \subseteq^* \bigsqcup_{i \leq e} R_i \cup \bigcup_{i \leq
  e} D_i$.
Since $D_e\subset D_{e+1}$ and
$(D_{e+1}-D_e)\cap \bigsqcup_{i\le e} R_i=\emptyset$, it follows that
$D_{e+1} \cap \tilde{D} \subseteq^* D_e$. By definition of a
generating set, there is an $l$ such that
\begin{equation}\label{E:Type9notType7}\bigsqcup_{i \leq e+1} R_i \cup
  \bigcup_{i \leq e+1} D_i
  \subseteq^* \tilde{D} \cup \bigsqcup_{i\le l} \tilde{R}_i.
\end{equation}
Similarly, there is a $k$ such that
$\tilde{D} \cup \bigsqcup_{i\le l} \tilde{R}_i \subseteq^*
\bigsqcup_{i \le k} R_i \cup \bigcup_{i \le k} D_i$.
By construction, $R_{k+1} - \breve{D}$ is infinite. Since $R_{k+1}$ is
disjoint from $\bigsqcup_{i\le e} R_i$, there is an $m > l$ such that
$(\tilde{R}_m\cap \overline{\bigsqcup_{i\le e} R_i}) - \breve{D}$ is
infinite.  By (\ref{E:Type9notType7}),
$\tilde{R}_m \cap D_{e+1} \subseteq^* \tilde{D}$. Since
$D_{e+1} \cap \tilde{D} \subseteq^* D_e$,
$\tilde{R}_m \cap (D_{e+1} - D_e) =^* \emptyset$.  By requirement
$\mathcal{I}_e$, $D_e\cap \overline{\bigsqcup_{i\le e} R_i}$ is
{small} inside $D_{e+1}\cap \overline{\bigsqcup_{i\le e} R_i}$.
So, by smallness, the infinite set
$$(\tilde{R}_m\cap \overline{\bigsqcup_{i\le e} R_i})-(D_{e+1}\cap
\overline{\bigsqcup_{i\le e} R_i})=(\tilde{R}_m\cap
\overline{\bigsqcup_{i\le e} R_i})-D_{e+1}$$
is c.e., contradicting Lemma~\ref{sec:one-needed-lemma}.  Thus, $A$
does not have Type 7.

\subsubsection{Type 9 not Type 8}

Lastly, we show that the $\mathcal{D}$-maximal set $A$ obtained via
the Type 9 construction is not Type 8. Assume that
$\{ \tilde{D}_0 , \tilde{D}_1, \ldots \tilde{R}_0, \tilde{R}_1, \ldots
\}$
is a Type 8 generating set for $\mathcal{D}(A)$. We may assume that
this generating set satisfies the properties in
Lemma~\ref{sec:types-7-8}. By the $\mathcal{R}'_e$ requirements and
the definition of generating set, we have the following facts.  There
is an $l$ such that
$D_0 \subseteq^* \bigsqcup_{i \leq l} \tilde{R}_i \cup \bigsqcup_{i
  \leq l} \tilde{D}_i$.
Then, there is a $k$ such that
$$\bigsqcup_{i \leq l} \tilde{R}_i \cup \bigsqcup_{i \leq l}
\tilde{D}_i \subseteq^* \bigsqcup_{i \leq k} R_i \cup \bigcup_{i \leq
  k} D_i.$$ Next, there is an $m>l$ such that
$$\bigsqcup_{i \leq k+1} R_i \cup \bigcup_{i \leq k+1} D_i \subseteq^*
\bigsqcup_{i \le m}\tilde{D}_i \cup \bigsqcup_{i \le m} \tilde{R}_i.$$
Finally, there is a $r>k+1$ such that
$$ \bigsqcup_{i \le m}\tilde{D}_i \cup \bigsqcup_{i \le m} \tilde{R}_i
\subseteq^* \bigsqcup_{i \le r} R_i \cup \bigcup_{i \le r} D_i.$$
By construction, $R_{r+1} - \breve{D}=R_{r+1} - D_{r+1}$ is
infinite. There is also an $n> m$ such that
$R_{r+1} \subseteq^* \bigsqcup_{i \le n}\tilde{D}_i \cup \bigsqcup_{i
  \le n} \tilde{R}_i$.
Hence, there is an $\tilde{m} > m$ such that
$\tilde{R}_{\tilde{m}} \cap (R_{r+1}- \breve{D})$ is infinite or
$\tilde{D}_{\tilde{m}} \cap (R_{r+1} - \breve{D})$ is infinite. In the
latter case,
$\tilde{D}_{\tilde{m}} - \bigsqcup_{i \leq l} \tilde{R}_i$ is an
infinite c.e.\ set disjoint from $D_0$ but not contained in
$\breve{D}$, contradicting Lemma~\ref{sec:one-needed-lemma}. So, the
former holds.
By the choice of $l$ and $m$,
$(D_{k+1} -D_k)\cap \tilde{R}_{\tilde{m}}\subseteq^*\bigsqcup_{l< i
  \le m} \tilde{D}_i$.
Let
$Y = \tilde{R}_{\tilde{m}} \cap \bigsqcup_{l< i \le m}
\tilde{D}_i$.
Since $\{\tilde{D}_i\}_{i\in\omega}$ consists of pairwise disjoint
sets, $Y$ is a c.e.\ set such that $D_0 \cap Y =^* \emptyset$. Now
$\tilde{R}_{\tilde{m}} \cap (D_{k+1} - D_k) \subseteq^* Y$, so
certainly
$(\tilde{R}_{\tilde{m}}\cap R_{r+1}) \cap (D_{k+1} - D_k) \subseteq^*
Y$.
Since $D_k\cap \overline{\bigsqcup_{j\le k} R_j}$ is small in
$D_{k+1}\cap\overline{\bigsqcup_{j\le k} R_j}$ by requirement
$\mathcal{I}_k$, the set
$$Y \cup [(\tilde{R}_{\tilde{m}}\cap R_{r+1})
-(D_{k+1}\cap\overline{\bigsqcup_{j\le k} R_j})]$$
is a c.e.\ set.  Note that $r+1>k$.  This set is disjoint from $D_0$
since $Y$ is and since
$R_{r+1}\subset\overline{\bigsqcup_{j\le k} R_j}$.  Moreover, this
c.e.\ is infinite since it contains
$\tilde{R}_{\tilde{m}} \cap (R_{r+1}- \breve{D})$, contradicting
Lemma~\ref{sec:one-needed-lemma}.  Hence, $A$ is a Type 9
$\mathcal{D}$-maximal set.

\subsection{  Infinitely many orbits of   $\mathcal{D}$-maximal sets of Types 7, 8, 9}\label{sec:break-these-mathc}
By Lemma \ref{sec:gener-sets-simple}, two automorphic sets share the
same Type.  We show here, however, that the collection of
$\mathcal{D}$-maximal sets of Type 7 (respectively Type 8, Type 9)
breaks into infinitely many orbits.  Specifically, for each of these
Types, we construct infinitely many pairwise nonautomorphic
$\mathcal{D}$-maximal sets of the given Type.  For each of these
Types, we will take two boolean algebras
${\mathcal{B}} = \join_{i\in\omega} {\mathcal{B}}_{b_i}$ and
$\tilde{\mathcal{B}}=\join_{i\in\omega}
{\tilde{\mathcal{B}}}_{\tilde{b}_i}$
(with computable skeletons $\{b_i\}_{i\in\omega}$ and
$\{\tilde{b}_i\}_{i\in\omega}$ respectively).  We then will consider
the $\mathcal{D}$-maximal sets $A$ and $\tilde{A}$ obtained via the
given Type construction based on $\mathcal{B}$ and
$\tilde{\mathcal{B}}$ respectively.  Each of $A$ and $\tilde{A}$ will
have a generating set of the appropriate Type, denoted as usual with
the sets in the generating set for $\mathcal{D}(\tilde{A})$ marked
with tildes.
We suppose that $\Phi:\mathcal{E}^*\to \mathcal{E}^*$ is an
automorphism with $\Phi(\tilde{A})=A$, i.e., $\tilde{A}$ and $A$ are
automorphic. For notational simplicity, we denote $\Phi(\tilde{W})$ by
$\hat{W}$ for any c.e.\ set $\tilde{W}$.

\subsubsection{Type 7}\label{SS:Type7InfOrbits}
First, suppose that $A$ and $\tilde{A}$ are Type 7.  Since ${A}$ is
$\mathcal{D}$-maximal, there exists an $l$ such that
$\Phi(\tilde{D})= \Dhat\subseteq^*D \cup \bigsqcup_{i \le l} R_i$. By
construction and Corollary~\ref{sec:hhsimple-like-typesBA},
$\join_{i > l} {\mathcal{B}}_{b_i}$ is a subalgebra of
$\tilde{\mathcal{B}}$.  This containment is not possible if, for some
$i>l$, the Cantor Bendixson rank of $\mathcal{B}_{b_i}$ is greater
than the rank of $\tilde{\mathcal{B}}$.

We leave it to the reader to construct, for all $j\in\omega$, a
computable boolean algebra $\mathcal{B}_j$ equipped with a computable
skeleton $\{{b_{j,i}}\}_{i\in\omega}$ (i.e.,
${\mathcal{B}}_j = \join_{i\in\omega} {\mathcal{B}}_{b_{j,i}}$) such
that, for all $i\in\omega$, the rank of $\mathcal{B}_{b_{j+1,i}}$ is
larger than the rank of $\mathcal{B}_j$.  By the argument above, this
collection of boolean algebras gives rise to an infinite collection of
pairwise nonautomorphic $\mathcal{D}$-maximal Type 7 sets.

\subsubsection{Type 8}

Now suppose that $A$ and $\tilde{A}$ are Type 8.  Since $A$ is
$\mathcal{D}$-maximal, there is an $l$ such that
$\Dhat_0 \subseteq^* \bigsqcup_{i\le l} D_i \cup \bigsqcup_{i\le l}
R_i$.
Similarly, there is an $n$ such that
$$\bigsqcup_{i\le l} D_i \cup \bigsqcup_{i\le l} R_i \subseteq^*
\bigsqcup_{i\le n} \Dhat_i \cup \bigsqcup_{i\le n} \Rhat_i.$$
For $m> n$, inside $\Rhat_m$, there is a hhsimple set $\Hhat$ of
flavor $\tilde{\mathcal{B}}_{\tilde{b}_m}$ such that
$\Hhat = \Rhat_m \cap \bigsqcup_{i\leq m} \Dhat_i$.  Also, $\Dhat_0$
is a Friedberg split of $\Hhat$ by construction.  Fix a $k>l$ such
that
$\Rhat_m\subseteq^*\bigsqcup_{i\le k} D_i \cup \bigsqcup_{i\le k}
R_i$.

We will explore what $\Hhat$ and $\Rhat_m$ look like. First, note that
for all $i\le l$,
$\Rhat_m \cap R_i \subseteq^* \Rhat_m \cap \bigsqcup_{i\le n} \Dhat_i
\subseteq^* \Hhat$.
Similarly, for all $i \le l$, $ \Rhat_m \cap D_i\subseteq^*\Hhat$.
Since $\Dhat_0 \cap \Rhat_m$ is a Friedberg split of $\Hhat$ and, for
$l < i \le k$, $\Dhat_0$ and $D_i$ are almost disjoint,
$(\Rhat_m - \Hhat) \cap \bigsqcup_{ l < i \le k} D_i =^* \emptyset$.
Therefore, there is at least one $r$ such that $l < r \le k$ and
$(\Rhat_m - \Hhat) \cap R_r$ is infinite. Let $F$ be the finite set of
all such $r$. For all $r \in F$ and $i \le k$, we have that
$D_i \cap R_r \cap \Rhat_m\subseteq^* R_r \cap \Hhat$. So,
$\tilde{\mathcal{B}}_{b_m}$ is a subalgebra of
$\join_{r \in F} \mathcal{B}_{b_r}$.  This is impossible if the rank
of $\tilde{\mathcal{B}}_{b_m}$ is greater than the rank of
$\join_{r \in F} \mathcal{B}_{b_r}$.

We again leave it to the reader to construct infinitely many
computable boolean algebras ${\mathcal{B}}_j$ each equipped with a
computable skeleton $\{b_{j, i}\}_{i\in\omega}$ such that
${\mathcal{B}}_j = \join_{i\in\omega} {\mathcal{B}}_{b_{j,i}}$ and the
rank of $\mathcal{B}_{b_{j+1,i}}$ is larger than the rank of the join
of finitely many $\mathcal{B}_{b_j,z}$. In fact, the collection of
boolean algebras from the Type 7 case in \S\ref{SS:Type7InfOrbits}
suffices.

\subsubsection{Type 9}

We assume the same setup as for the Type 8 case but for sets of Type
9. As above, there exist $l$ and $n$ such that
\begin{equation*}
  \Dhat_0 \subseteq^* \bigcup_{i\le l} D_i \cup \bigsqcup_{i\le l} R_i \subseteq^* \bigcup_{i\le n} \Dhat_i \cup
  \bigsqcup_{i\le n} \Rhat_i.
\end{equation*}
For $m> n$, inside $\Rhat_m$, there is a hhsimple set
$\Hhat = \Rhat_m \cap \Dhat_m$ of flavor
$\tilde{\mathcal{B}}_{\tilde{b}_m}$.
Let $k>l$ be such that
$\Rhat_m\subseteq^*\bigcup_{i\le k} D_i \cup \bigsqcup_{i\le k} R_i$.
As before, for all $i\le l$,
$\Rhat_m \cap (R_i\cup D_i) \subseteq^* \Rhat_m \cap \bigcup_{i\le n}
\Dhat_i \subseteq^* \Hhat$.

At this point, the argument differs.  By Lemma \ref{D_esmD_{e+1}}
(\ref{D_esmD_{e+1}1}), $D_0\cap R_r$ almost equals or is small major
in $D_r\cap R_r$ for any $r$. So, for any $r$, if
$\Rhat_m \cap R_r\subseteq^*D_r \cap R_r$, then
$\Rhat_m \cap R_r\subseteq^* D_0 \cap R_r$.  In other words, if
$(\Rhat_m \cap R_r) - D_r$ is finite,
$\Rhat_m \cap R_r\subseteq^* D_0 \cap R_r\subseteq^* \Hhat$.  By Lemma
\ref{D_esmD_{e+1}} (\ref{D_esmD_{e+1}2}),
$D_0\cap \overline{\bigsqcup_{i\le j}R_i}$ almost equals or is small
major in $D_j\cap \overline{\bigsqcup_{i\le j}R_j}$.  By choice of
$k$,
$\Rhat_m \cap \overline{\bigsqcup_{i\le k}R_i}\subseteq^*D_k \cap
\overline{\bigsqcup_{i\le k}R_i}$.
So, similarly,
$\Rhat_m \cap \overline{\bigsqcup_{i\le k}R_i}\subseteq^* D_0 \cap
\overline{\bigsqcup_{i\le k}R_i}$,
and, hence, $\Rhat_m - \bigsqcup_{i\le k} R_i\subseteq^*\Hhat$.

Let $F$ be the set
of $r
\le k$ such that $(\Rhat_m \cap R_r) -
D_r$ is infinite.  The statements in the previous paragraph together
with the fact that $\Rhat_m-
\Hhat$ is infinite imply that $F$ is nonempty and that $\Rhat_m -
\bigsqcup_{r \in F} R_r \subseteq^* \Hhat$.
Recall that $D_0\cap
R_r$ equals $D_r\cap R_r$ or is small major in $D_r\cap
R_r$.  In the latter case, $D_0\cap R_r\cap
\Rhat_m$ equals or is small major in $D_r\cap R_r\cap
\Rhat_m$ by Lemma \ref{stob1} (\ref{S26}).  Since $D_0\cap R_r\cap
\Rhat_m\subseteq^*\Hhat \cap R_r$, in any of these cases, $\Hhat \cap
R_r$ must almost contain $D_r \cap R_r\cap
\Rhat_m$.  In particular, if $D_0\cap R_r\cap
\Rhat_m$ is major in $D_r\cap R_r\cap \Rhat_m$, $\Hhat \cap
R_r$ almost contains $D_r \cap R_r\cap
\Rhat_m$ by Lemma~\ref{stob1} (\ref{S25}) since $\Hhat \cap
R_r$ is hhsimple.  So,
$\tilde{\mathcal{B}}_{b_m}$
is a subalgebra of $\join_{r
  \in F}
\mathcal{B}_{b_r}$.  But if the rank of
$\tilde{\mathcal{B}}_{b_m}$
is greater than the rank of $\join_{r
  \in F}
\mathcal{B}_{b_r}$ this cannot occur. The collection of Boolean
Algebras from the last section demonstrates that the collection of
sets of Type 8 breaks up into infinitely many orbits.

\subsection{Questions on the orbits of Type 7, 8,  9 $\mathcal{D}$-maximal sets}\label{sec:anything-know-about}
We know nothing about the structure of the infinitely many orbits
containing Type 7, 8, or 9 $\mathcal{D}$-maximal sets.  Recall that,
by Corollary \ref{sec:hhsimple-like-typesBA}, each set of Type 7, 8,
or 9 is associated with a boolean algebra $\mathcal{B}$ (which depends
on a choice of generating set).  We think of the input boolean algebra
to our construction as a partial invariant for the resulting
$\mathcal{D}$-maximal sets of Type 7, 8, and 9. Suppose $\mathcal{B}$
is a computable boolean algebra with a computable skeleton.
If $A$ is the $\mathcal{D}$-maximal set resulting from our
construction with input $\mathcal{B}$ and $\tilde{A}$ is automorphic
to $A$, Corollary~\ref{sec:hhsimple-like-typesBA} and Lemma
\ref{sec:gener-sets-simple} imply that $\tilde{A}$ is hhsimple-like
but we do not know if the assoicated boolean algebra is isomorphic to
$\mathcal{B}$, only that they are ``similar'' rank.
These observations lead to the following question.

  \begin{question}\label{Q1:Orbits7-9}
    Suppose that the $\mathcal{D}$-maximal sets $A$ and $\tilde{A}$,
    both of Type 7, 8, or 9, are associated with the boolean algebras
    $\mathcal{B}$ and $\tilde{\mathcal{B}}$ respectively.  If
    $\mathcal{B}$ and $\tilde{\mathcal{B}}$ are isomorphic (or have
    the same or ``similar'' rank), are $A$ and $\tilde{A}$
    automorphic?
  \end{question}

  We make a few comments about Question \ref{Q1:Orbits7-9}.  We begin
  with the Type 7 case.  Let $A$ and $\tilde{A}$ be Type 7
  $\mathcal{D}$-maximal sets.  Suppose that $\{ D, R_0, R_1 \ldots \}$
  is the generating set for $\mathcal{D}(A)$ and that
  $\mathcal{D}(\tilde{A})$ has a generating set of the same form with
  all sets marked by tildes.  Finally, assume that $A \sqcup D$ and
  $\tilde{A} \sqcup \tilde{D}$ are both hhsimple sets of flavor
  boolean algebra $\mathcal{B}$. So, by \citet{Maass:84}, $A \sqcup D$
  and $\tilde{A} \sqcup \tilde{D}$ are automorphic, but we do not know
  whether $A$ and $\tilde{A}$ are automorphic. 
  A direct approach would be to use an extension theorem to map $D$ to
  $\tilde{D}$ and the $R_i$ to the $\tilde{R}_i$. We can take
  computable subsets of $D$ to computable sets of $\tilde{D}$. But it
  is not clear how to ensure that $D\cap R_i$ is taken to
  $\tilde{D} \cap \tilde{R}_i$.  It seems possible that this could be
  done by directly building the isomorphism.  If an isomorphism could
  be built in the Type 7 case, we speculate that an isomorphism could
  be built in the more complicated Type 8.  However, the Type 9 case
  seems fundamentally more difficult.  In that case, one needs to
  ensure that $D_{i+1}$ automorphic to $\tilde{D}_{i+1}$ via an
  automorphism taking $D_i$ to $\tilde{D}_i$. This is seems beyond the
  limits of current extension theorem technology.

  Note that the above comments only apply to $\mathcal{D}$-maximal
  sets of Types 7, 8, and 9.  By Corollary
  \ref{sec:hhsimple-like-typesBA}, without the
  $\mathcal{D}$-maximality assumption, we only know that the boolean
  algebra $\mathcal{B}$ that corresponds to the sets of Types 7, 8,
  and 9 is a proper substructure of $\mathcal{L}(\breve{D})$. Hence,
  we have no insight into the question of when Type 7, 8, and 9 sets
  are automorphic.

  Finally, given a computable boolean algebra $\mathcal{B}$ with a
  computable skeleton, we will construct $\mathcal{D}$-maximal sets
  $A_0$ and $A_1$ of Types 7, 8, and 9 respectively of flavor
  $\mathcal{B}$ such that $A_0$ is complete and $A_1$ is not (see
  Remark~\ref{degree}). In addition to Question \ref{Q1:Orbits7-9}, we
  also leave unanswered whether the particular sets $A_0$ and $A_1$ we
  construct are automorphic.

\subsection{The Construction}\label{sec:construction}
We give the details of the construction of $\mathcal{D}$-maximal sets
of Types 7, 8, and 9.  We focus on the construction of Type 9
$\mathcal{D}$-maximal sets $A$ as this case is the most complicated,
and we leave the adjustments for the Type 7 and 8 cases to the reader.

We construct the set $A$ using a $\Pi_2^0$-tree argument that is very
similar to the $\Delta^0_3$-isomorphism method.  In this construction
our priority tree will just be $2^{<\omega}$, i.e., each requirement
on the tree can be met in one of two possible ways.  As usual we
define a stage $s$ computable approximation $f_s$ to the true path $f$
so that $f=\liminf_s f_s$ where the value of $f(n)$ indicates how the
$n$-th requirement is satisfied.  In our situation, $f(n)=0$ will
indicate that a certain set related to the $n$-th requirement is
infinite.  The advantage of the tree construction over the usual
priority argument is that our strategy for meeting the $n$-th
requirement can depend on how the requirements $i<n$ were met.
Elements will be placed at nodes on the tree to aid in meeting these
requirements.  We intuitively refer to elements as balls, since their
location can change throughout the construction.  We view our tree as
growing downward since balls mainly move down through the tree.  We
say a node $\alpha$ is {\em visited} at stage $s$ if
$\alpha\preccurlyeq f_s$ and $\alpha$ is {\em reset} at stage $s$ if
$f_s<_L \alpha$ where $<_L$ is the lexiographic ordering on
$2^{<\omega}$.

At each node $\alpha \in 2^{<\omega}$, we attempt to build a
computable set $R_{\alpha}$ and c.e.\ set $D_\alpha$.  For $\lambda$
the empty node, the resulting $D_\lambda$ is $A$, and we set
$R_\lambda = \emptyset$. We build these sets so that the collection
$\{D_{\alpha}\mid\lambda\not=\alpha\prec f \}\cup\{R_{\alpha}\mid
\lambda\not=\alpha\prec f \} $
is a generating set for $\mathcal{D}(A)$.  We ensure that $R_{\alpha}$
is computable for $\alpha\prec f$ by enumerating the set
$\overline{R}_\alpha$ as well.  Specifically, at each node
$\alpha \in 2^{<\omega}$, we construct a set $\widetilde{R}_\alpha$ so
that $\widetilde{R}_\alpha=^*\overline{R}_\alpha$ if $\alpha\prec f$.
Once an element enters any $R_\alpha$, $D_\alpha$, or
$\widetilde{R}_\alpha$, it remains there.  So, these are all c.e.\
sets.  Moreover, no element enters any of these sets before the
element has been placed on the tree.

We recast the requirements $\mathcal{S}_e$ and $\mathcal{R}_e^\prime$
in this tree language.  For $\alpha\in 2^{<\omega}$ with
$|\alpha| = e$, we have the requirements:
\begin{equation*}
  \tag*{$\mathcal{S}_{\alpha}$:}
  \overline{A} \neq W_e,
\end{equation*}

\begin{equation*}
  \tag*{$\mathcal{R}^\prime_{\alpha}$:}
  W_e \subseteq^*  \bigcup_{\beta \preccurlyeq \alpha} D_\beta  \cup
  \bigsqcup_{\beta \preccurlyeq \alpha} R_\beta
  \text{ or } {W_e} \cup \bigcup_{\beta \preccurlyeq \alpha} D_\beta \cup
  \bigsqcup_{\beta \preccurlyeq \alpha} R_\beta  =^* \omega.
\end{equation*}
We will address the requirements $\mathcal{L}_e$ and $\mathcal{I}_e$
after we describe how to meet the above requirements. First, we
describe the general rules about how balls move down the tree.  The
outcomes and action for $\mathcal{R}_\alpha^\prime$ also control this
movement (and maintain our construction guarantees), but we delay
these details until \S\ref{ActionofR'}.
Given $\beta\in 2^{<\omega}$, we let $\beta^-$ denote the node
immediately preceding $\beta$.

The position function $\alpha(x, s)$ is the location of an element $x$
on the tree $2^{<\omega}$ at stage $s$.  Elements on the tree either
move downward from the root $\lambda$ by gravity or are pulled
leftward by action for requirement $\mathcal{R}_\alpha^\prime$.
Meanwhile, the requirement $\mathcal{S}_\alpha$ restrains movement
down the tree while it secures a witness denoted
$x_\alpha$. 
We say that $x$ is {\em $\beta$-allowed at stage $s$}
if 
$x > |\beta|$, 
$x$ is not in
$\bigsqcup_{\gamma \preccurlyeq\beta} R_\gamma\cup \bigcup_{\gamma
  \preccurlyeq \beta} D_\gamma$
and $x$ has been enumerated into $\widetilde{R}_\gamma$ for all
$\gamma\preccurlyeq \beta$. By induction on $\beta \prec f$, almost
all balls not in
$\bigsqcup_{\gamma \preccurlyeq\beta} R_\gamma\cup \bigcup_{\gamma
  \preccurlyeq \beta} D_\gamma$ are $\beta$-allowed.

Given $f_s$, we determine the position function $\alpha(x,s)$ by the
following rules (defined stagewise).  
 At stage $s$, the ball $s$ enters the tree and is
placed on node $\lambda$, i.e., we set $\alpha(s, s)=\lambda$, and we
enumerate $s$ into $\widetilde{R}_\lambda$.  Hence, $s$ is
$\lambda$-allowed.
The node $\beta$ may pull any $x$ for $\mathcal{R}_\beta^\prime$ at
stage $s$ if $\beta \le _L \alpha(x,s-1)$, $x$ is
$\alpha(x,s-1) \cap f_s$-allowed, and, for all stages $t$, if
$x \leq t \leq s$, then $\beta \leq_L
f_t$. 
 In this case, move $x$ to $\beta$, i.e., set 
 $\alpha(x,s) = \beta$, and enumerate $x$ into $\widetilde{R}_\gamma$ for all $\gamma$ such that $ \alpha(x,s-1) \cap f_s\prec\gamma\prec \beta$.  
For details on  when a ball  is pulled and what action is taken with pulled balls, see Remark~\ref{sec:acti-mathc}.

On the other hand, suppose that $x$ is $\beta^-$-allowed for some $\beta\preccurlyeq f_s$, $x$ is not
the current witness $x_{\beta^-}$ for $\mathcal{S}_{\beta^{-}}$, and,
for all stages $t$, if $x \leq t \leq s$, then $\beta \leq_L f_t$.  In
this case, move $x$ to $\beta$ at stage $s$ so that
$\alpha(x,s) = \beta$.  If an element $x$ on the tree is not moved by
these rules 
 and $\alpha(x, s-1)$ is
not reset at stage $s$, set $\alpha(x, s)=\alpha(x, s-1)$.  If
$\alpha(x, s-1)$ is reset at stage $s$, let
$\alpha(x, s)=\alpha(x,s-1) \cap f_s$.

Note that, throughout the construction, we only move $x$ to some node
$\beta$ at stage $s$ (i.e., set $\alpha(x, s)=\beta$) if   $\beta\preccurlyeq f_s$ or
$\beta$ pulled $x$ (in which case, there was an earlier stage $t$ such
that $\beta \preccurlyeq f_t$, and $\beta$ has not been reset since stage $t$) and, after $x$ is moved to $\beta$ at stage $s$, the ball $x$ is
(at least) $\beta^-$-allowed.
In addition, by the action for $\mathcal{R}'_{\alpha}$ described in
\S\ref{ActionofR'}, we will ensure the following if $\alpha\prec f$.
First, infinitely many balls will reach $\alpha$ and be
$\alpha$-allowed.  Second, for each ball that is $\alpha$-allowed at
node $\alpha$, we add another ball to $R_\alpha$. Third, all but
finitely many balls are enumerated into $R_\alpha$ or
$\widetilde{R}_\alpha=^*\overline{R}_\alpha$ and each of these sets is
infinite.  We now describe the details of each requirement's action.

\subsubsection{Action for $\mathcal{S}_{\beta}$} \

\noindent
{\em Assigning witnesses to $\mathcal{S}_{\beta}$.}\quad
We meet $\mathcal{S}_{\beta}$ in the usual way.  For any
$\beta\in 2^{<\omega}$, we let $x_{\beta, s}$ denote the stage $s$
witness for $\mathcal{S}_{\beta}$.  The witness $x_{\beta, 0}$ is
undefined.  Suppose that
$W_{|\beta|, s}\cap D_{\lambda, s}=\emptyset$, witness $x_{\beta, s}$
is undefined, and there is a stage $t>s$ and an element
$x\ge 2|\beta|$ such that $\beta\preccurlyeq f_t$ and
$\alpha(x, t)=\beta$.  At the least such stage $t>s$, define
$x_{\beta, t}$ to be the least $x$ such that $\alpha(x,
t)=\beta$. 
Once $x_{\beta, t}$ is defined, we let $x_{\beta, t'}=x_{\beta, t}$
unless $f_{t'} <_L \beta$ for $t'>t$.  In this case, we make
$x_{\beta,t'}$ undefined at that stage.
The node $\beta$ may not take any action while $x_{\beta,s}$ is
undefined and $W_{|\beta|, s}\cap D_{\lambda, s}=\emptyset$.

\noindent
\emph{Placing witnesses into $D_\lambda$.}\quad Suppose
$\alpha \preccurlyeq f_s$, $|\alpha| = e$,
$W_{e,s} \cap D_{\lambda,s} = \emptyset$, and there is an
$x_{\beta,s}\ge 2e$ such that $|\beta| = |\alpha| = e$ and
$x_{\beta,s} \in W_{e, s}$.  Then, enumerate $x_{\beta,s}$ into
$D_\lambda$ and $\widetilde{R}_\gamma$ for all $\gamma\in 2^{<\omega}$
and remove $x_{\beta,s}$ from the tree. This is the only way balls
enter $D_\lambda$.

Suppose that $\alpha\prec f$ and $\overline{D}_{\lambda} = W_e$ where
$|\alpha|=e$. By the assumption that infinitely many balls will reach
$\alpha$, it is straightforward to show that some witness
$x_{\beta, s}\in W_e$ for $|\beta|=|\alpha|=e$ is enumerated into
$D_\lambda$ to meet $\mathcal{S}_{\alpha}$, a contradiction.  As
usual, $\mathcal{S}_{\alpha}$ acts at most once (and at most one
$\mathcal{S}_{\beta}$ acts for a given $e=|\beta|$) and $D_\lambda$ is
coinfinite since any witness for $\beta\in 2^e$ satisfies
$x_{\beta,s}\ge 2|\beta|=2e$.  Note that a ball might enter $W_e$ long
before it becomes our witness.  So this action does not imply that $A$ is
promptly simple.

\begin{remark}\label{degree}
  Our action for $\mathcal{S}_\beta$ mixes with both finite permitting
  and coding. 
  For permitting, we ask for permission when we want to place a ball
  into $D_\lambda$. If we get permission, then we add the ball to
  $D_\lambda$. While waiting for permission, we set up a new ball as
  another witness $x_{\beta}$. If enumerating that ball into
  $D_\lambda$ would also satisfy $\mathcal{S}_\beta$, we ask again for
  permission. Under finite permitting, we will eventually receive
  permission to enumerate some witness for $\mathcal{S}_\beta$ into
  $D_\lambda$. Hence, we can construct $D_\lambda$ to be incomplete or
   computable in any noncomputable c.e.\ set.

  Fix a c.e.\ set $W$ such as $K$.  To code $W$ into $D_\lambda$, when
  $W$ changes below $e$ at stage $s$, dump all currently defined
  witnesses $x_{\beta,s}$ for $|\beta| \geq e$, into $D_\lambda$.
  To determine $W$ below $e$, wait until there is a witness
  $x_{\alpha,s}$ not in $D_\lambda$ for $|\alpha| \ge e $. (Since the
  empty set has infinitely many indices and $\lim_s x_{\alpha,s}$
exists for all $\alpha \prec f$, we will always find such a
$x_{\alpha,s}$.)  Then, $W$ below $e$ will not change after stage
$s$. So, we can construct $D_\lambda$ to be complete.

These remarks also apply to the construction of a set with an
$A$-special list in \citet[Section 7.2]{MR2366962}. \end{remark}

\subsubsection{Action for
  $\mathcal{R}^\prime_{\alpha}$}\label{ActionofR'}

To meet $\mathcal{R}^\prime_{\alpha}$, we need to know whether a
certain c.e.\ set is infinite. For $e =|\alpha|$, we define the set
\begin{equation*}
  \tilde{W}_{e} = \{ x \mid (\exists s)[ x \text{ is $\alpha^-$-allowed
    at or before stage }s\ \&\ x \in W_{e,s}]
  \}.
\end{equation*}
The action for $\mathcal{R}^\prime_{\alpha}$ depends on whether the
c.e.\ set
\begin{equation*}
  X_{\alpha^-} =
  \tilde{W}_e \diagdown
  (\bigsqcup_{\beta \prec \alpha} R_\beta
  \cup \bigcup_{\beta \prec \alpha} D_\beta)
\end{equation*}
is infinite. Notice that $\tilde{W}_e$ and $X_{\alpha^-}$ depend only
on nodes that are proper subnodes of $\alpha$.  By definition,
a ball that is $\alpha^-$-allowed at stage $s$ is   not in
$\bigsqcup_{\beta \prec \alpha} R_\beta \cup \bigcup_{\beta \prec
  \alpha} D_\beta$ at stage $s$.
Recall our promise that $\alpha^-\prec f$ implies that infinitely many
balls will be $\alpha^-$-allowed at some point.  Hence, $X_{\alpha^-}$ is infinite
if and only if infinitely many $\alpha^-$-allowed balls enter $W_e$
before they enter
$\bigsqcup_{\beta \prec \alpha} R_\beta \cup \bigcup_{\beta \prec
  \alpha} D_\beta$.

Each $\alpha$ in the tree encodes a guess as to whether $X_{\alpha^-}$
is infinite.  In particular, $\alpha(|\alpha|-1)=0$ indicates the
guess that $X_{\alpha^-}$ is infinite.  The statement that the c.e.\
set $X_{\alpha^-}$ is infinite is $\Pi_2^0$, so this information can
be coded into a tree in the standard way.  Specifically, we can define
the true path $f$ and the stage $s$ approximation to the true path
$f_s$ so that $\alpha$ encodes a correct guess if
$\alpha \preccurlyeq f$.  Since these definitions are standard, we
leave them to the reader. Similar constructions with all the details
can be found in \cite{mr95f:03064} and \cite{Weber:06}.

We define a helper set $P_\alpha$ based on the guess encoded by
$\alpha$.  If $\alpha$ encodes the guess that $X_{\alpha^-}$ is
infinite, we let $P_\alpha= X_{\alpha^-}$. Otherwise, we let
$P_\alpha = \omega \diagdown (\bigsqcup_{\beta \prec \alpha} R_\beta
\cup \bigcup_{\beta \prec \alpha} D_\beta)$.
If $X_{\alpha^-}$ is in fact finite, then $W_e$ is almost contained in
$ \bigsqcup_{\beta \prec \alpha} R_\beta \cup \bigcup_{\beta \prec
  \alpha} D_\beta$,
and $\mathcal{R}^\prime_\alpha$ is met.  We describe the action for
$\mathcal{R}^\prime_\alpha$ and show that $\mathcal{R}^\prime_\alpha$
is also met if $X_{\alpha^-}$ is infinite and $\alpha\prec f$.

\begin{remark}[Pulling]\label{sec:acti-mathc}
  If $\alpha\preccurlyeq f_s$ and $x_{\alpha, s}$ is defined at stage
  $s$, then $\alpha$ pulls, possibly at later stages, the least
  available balls that are greater than $|\alpha|$
  and in $P_\alpha$ for $\mathcal{R}^\prime_\alpha$ until it has
  secured two such balls $x$ and $y$. After such a time, $\alpha$
  cannot pull again until $\alpha$ is once more on the approximation of the true
  path.  Any ball may be pulled at most once by a given node
  $\alpha$.
\end{remark}

If $\mathcal{R}^\prime_\alpha$ has secured two balls
$x, y\in P_\alpha$ with
$\alpha(x,s)=\alpha(y,s)=\alpha\preccurlyeq f_s$, we enumerate $x$
into $\widetilde{R}_\alpha$, so that $x$ is $\alpha$-allowed at stage
$s$,
and enumerate $y$ into $R_{\alpha,s}$.  If there are any other balls
$z$ such that $\alpha(z,s)=\alpha$, we enumerate these balls into
$R_{\alpha, s}$.  Some of these balls might be in some $D_\beta$ where
$\beta \prec \alpha$.
For any $\beta$, if a ball is added to $R_\beta$, then also add it to
$\widetilde{R}_\gamma$ for all $\gamma$ extending $\beta$. By
construction, if $\alpha\prec f$, the only balls not in $R_\alpha$ or
$\widetilde{R}_\alpha$ are the balls $x$ such that
$\alpha(x,s)<_L \alpha$ or $x$ is one of finitely many unused
potential witnesses for $\mathcal{S}_\beta$ with
$|\beta|\le |\alpha|$.  Hence, $R_\alpha$ is computable.

Suppose that $\alpha\prec f$.  Since $P_\alpha$ is infinite and all
but finitely many balls pass through $\alpha$, there are infinitely
many stages $s$ such that $\alpha\prec f_s$ and the node $\alpha$
holds two balls in $P_{\alpha^-}$ for $\mathcal{R}^\prime_{\alpha}$.
Hence, infinitely many balls will reach $\alpha$ and be
$\alpha$-allowed.  Moreover, both $R_\alpha$ and
$\widetilde{R}_\alpha$ will be infinite.  By construction,
$ \bigsqcup_{\beta \preccurlyeq \alpha} R_\beta \cup \bigcup_{\beta
  \prec \alpha} D_\beta \cup P_\alpha =^* \omega$.
So, if $X_{\alpha^-}$ is infinite,
$ \bigsqcup_{\beta \preccurlyeq \alpha} R_\beta \cup \bigcup_{\beta
  \prec \alpha} D_\beta \cup W_e =^* \omega$.
Therefore, $\mathcal{R}^\prime_\alpha$ is met.

\subsubsection{Meeting the other requirements}

We divide $R_\alpha$ into two parts:
the balls that enter $R_\alpha$ before being placed in any $D_\beta$
for $\beta\prec\alpha$, specifically $R^+_\alpha=\bigcup_{\beta\prec\alpha} (R_\alpha\diagdown D_\beta)$, and the remaining balls
$R^-_\alpha=R_\alpha-R^+_\alpha$. Clearly,
$R^-_\alpha \subseteq \bigsqcup_{\beta \prec \alpha} D_{\beta}$.
Since the infinitely many pairs of balls pulled for
$\mathcal{R}^\prime_\alpha$ are not in $D_\beta$ for any
$\beta \prec \alpha$, $R^+_\alpha$ is infinite if $\alpha\prec f$.

Recall Lachlan's construction (Theorem \ref{T:Lachlan}) that for
$\mathcal{B}_e$ there is a hhsimple set of flavor
$\mathcal{B}_e$. Apply this construction to $R^+_\alpha$ to get $H_e$
and meet requirement $\mathcal{L}_{e}$.  For the Type 9 case, use
Lachlan's small major subset construction (Theorem
\ref{sec:smallmajorsubset}) to satisfy $\mathcal{I}_e$ and the
construction assumptions in \S\ref{Dmaxreq89}, i.e., build $D_e$ so
that $D_e\cap R_j=H_j$ for $j\le e$ and $D_e$ is small major in
$D_{e+1}$ on $\overline{\bigsqcup_{j\le e} R_j}$.  (For the Type 7
case, add all balls in $H_e$ into $D_{\lambda^+}$. For the Type 8
case, construct a Friedberg splitting $\bigsqcup_{j\leq e} H_{e,j}$ of
$H_e$ and add the balls in $H_{e,j}$ into $D_j$.)  This ends the
construction.\hfill $\square$

\bibliographystyle{plainnat}
\bibliography{dmax} \end{document}